\documentclass[11pt,a4paper]{amsart}

\newcommand{\1}{{1\!\!1}}
\usepackage{latexsym}

\usepackage{amssymb}
\usepackage{amsthm}
\usepackage{amsmath}
\usepackage{amstext}
\usepackage{amscd}

\numberwithin{equation}{section}

\theoremstyle{plain}
\newtheorem{thm}{Theorem}[section]
\newtheorem{prop}[thm]{Proposition}
\newtheorem{cor}[thm]{Corollary}
\newtheorem{lem}[thm]{Lemma}

\newtheorem{theorem*}{Theorem}[]

\theoremstyle{definition}

\newtheorem{example}[thm]{Example}

\theoremstyle{remark}
\newtheorem{rem}[thm]{Remark}

\newcommand{\R}{\mathbb{R}}

\newcommand{\Z}{\mathbb{Z}}

\newcommand{\T}{\mathbb{T}}

\newcommand{\V}{\mathbf{V}(\R)}
\newcommand{\Sch}{\mathbf{Sch}_{c} (\R)}
\newcommand{\Reg}{\mathbf{Reg}_{comp}(\R)}

\newcommand{\s}{\, \mathbf{s}\, }
\newcommand{\Cat}{\mathcal{C}}

\newcommand{\HCat}{Ho\, \mathcal{C}}
\newcommand{\Kdot}{{K_* }}
\newcommand{\A}{\mathcal{G}}
\newcommand{\AS}{\mathcal{AS}}
\newcommand{\Na}{\mathcal{N}}
\newcommand{\X}{\mathcal{X}}

\newcommand{\proj}{\mathbb{P}}

\renewcommand{\1}{\mathbf{1}}

\newcommand{\inv}{^{-1}}

\DeclareMathOperator{\sgn}{sign\,}

\DeclareMathOperator{\supp}{Supp}
\DeclareMathOperator{\codim}{codim}
\DeclareMathOperator{\Ker}{Ker}
\newcommand{\WC}{{\mathcal WC}}
\newcommand{\XR}{\underline X}
\newcommand{\YR}{\underline Y}
\newcommand{\ZR}{\underline Z}
\newcommand{\DR}{\underline D}
\newcommand{\UR}{\underline U}

\date{February 24, 2009}

\title{The weight filtration for real algebraic varieties}

\author{ Clint McCrory and Adam Parusi\'nski }

\thanks{Research partially supported by a Math\' ematiques en Pays de la Loire (MATPYL) grant.}

\address {Mathematics Department, University of Georgia, Athens GA
30602, USA}

\email{clint@math.uga.edu}

\address {Laboratoire Angevin de Recherche en Math\'ematiques, UMR 6093 
du CNRS,  Universit\'e d'Angers,
   2, bd Lavoisier, 49045 Angers cedex, France}
   
\email{adam.parusinski@univ-angers.fr}

\begin{document}

 \begin{abstract}
 Using the work  of Guill\'en and Navarro Aznar we associate
 to each real algebraic  variety a filtered chain
complex,
the weight complex, which is  well-defined up to 
filtered quasi-isomorphism, 
and which induces on Borel-Moore
homology
  with $\Z_2$ coefficients an analog of the
weight filtration for complex algebraic varieties.

The weight complex can be represented by a geometrically
defined  filtration on the
 complex of semialgebraic chains.  To show this we define the
weight complex for
 Nash manifolds and, more  generally, for arc-symmetric sets, and we adapt to
Nash manifolds the
theorem of  Mikhalkin  that  two  compact connected smooth
manifolds of the
same dimension can be connected by a sequence of smooth blowups and
blowdowns.

 The weight complex is acyclic for smooth blowups and additive for closed
inclusions.  As a corollary
 we obtain a new construction of the virtual Betti numbers, which are additive
invariants of real algebraic varieties, and we
 show their invariance by a large class of mappings that includes 
regular  homeomorphisms and
 Nash diffeomorphisms.
\end{abstract}

\subjclass[2000]{Primary: 14P25. Secondary: 14P10, 14P20}


\maketitle

The weight filtration of the homology of a real variety was introduced by Totaro \cite{totaro}. He
used the work of Guill\'en and Navarro Aznar \cite{navarro} to show the existence of such a filtration,
by analogy with Deligne's weight filtration for complex varieties \cite{deligne}, as
generalized by Gillet and Soul\'e \cite{gilletsoule}. There is also earlier unpublished work
on the real weight filtration by M. Wodzicki, and more recent unpublished work on weight filtrations
by Guill\'en and Navarro Aznar \cite{navarro2}.

Totaro's weight filtration for a compact  variety is associated to the spectral 
sequence of a cubical 
hyperresolution. For complex varieties this spectral sequence collapses with rational
coefficients, but for real varieties, where it is defined with $\Z_2$
coefficients, the spectral sequence does not collapse in general.
 We show, again using the work of Guill\'en and Navarro Aznar, that the
weight spectral sequence is itself a natural invariant of a real variety. There is a functor that assigns
to each real algebraic variety a filtered chain complex, the \emph{weight complex}, that is unique up to 
filtered quasi-isomorphism, and functorial for proper regular morphisms.
The weight spectral sequence is the spectral sequence associated to this filtered complex, 
and the weight
filtration is the corresponding filtration of Borel-Moore homology with coefficients in $\Z_2$.

To apply the extension theorems of Guill\'en and Navarro Aznar \cite{navarro}, we work in the category of schemes over $\R$, for which one has resolution of singularities, the Chow-Hironaka Lemma (\emph{cf.}\ \cite{navarro} (2.1.3)), and the compactification theorem of Nagata \cite{nagata}. We obtain the weight complex as a functor of schemes and proper regular morphisms.

Using the theory of Nash constructible functions, we extend the weight complex functor to Kurdyka's larger category of $\mathcal A\mathcal S$ (arc-symmetric) sets (\cite{kurdyka1},
\cite{aussois}), and we obtain in particular that the weight complex is invariant under regular rational homeomorphisms of real algebraic sets in the sense of Bochnak-Coste-Roy \cite{BCR}.
We give a geometric chain-level description of the weight filtration of the
complex of semialgebraic chains.

The characteristic properties of the weight complex describe how it behaves with respect to
generalized blowups (acyclicity) and inclusions of open subvarieties (additivity). The initial term
of the weight spectral sequence yields additive invariants for real algebraic varieties, the virtual
Betti numbers \cite{mccroryparusinski}. Thus we obtain that the virtual Betti numbers are invariants
of regular homeomorphims of real algebraic sets. For real toric varieties, the weight spectral sequence 
is isomorphic to the toric spectral sequence introduced by Bihan, Franz, McCrory, and van Hamel
\cite{BFMH}.

In section \ref{homological} we prove the existence and uniqueness of the filtered weight complex of a real algebraic
variety. The weight complex is the unique acyclic additive
extension to all varieties of the functor that assigns to a nonsingular projective variety the complex
of semialgebraic chains with the canonical filtration.

In section \ref{geometric} we characterize the weight filtration of the semialgebraic chain complex using resolution
of singularities.

In section \ref{nash} we introduce the Nash constructible filtration of semialgebraic chains, following
Pennanea'ch \cite{penn2}, and we show that it gives the weight filtration. A key tool is Mikhalkin's theorem
\cite{mikhalkin1} that any two connected closed $C^\infty$ manifolds of the same dimension
can be connected by a sequence of blowups and blowdowns.

In section \ref{toric} we show that for a real toric variety the Nash constructible filtration is the same as
the filtration on cellular chains defined by Bihan \emph{et al.}\ using toric topology.

We thank Michel Coste for his comments on a preliminary version of this paper.

\section{The homological weight filtration}
\label{homological}
\medskip

By a \emph{real algebraic variety} we mean a reduced separated scheme of finite type over $\R$.  
By a \emph{compact}  variety we mean a scheme that is complete (proper over $\R$).
We adopt the following notation of Guill\'en and Navarro Aznar \cite{navarro}. Let $\Sch$ be
the category of real algebraic varieties and proper regular morphisms, \emph{i.\ e.}\ proper  morphisms of schemes.   
By $\Reg$ we denote the 
subcategory of compact nonsingular varieties, and by $\V$ the category of projective nonsingular varieties.   A proper morphism or a compactification of varieties will always be understood in the scheme-theoretic sense.

In this paper we are interested in the topology of the set of real points of a real algebraic variety $X$.
Let $\XR$ denote the set of real points of $X$. The set $\XR$, with its sheaf of regular functions, is 
a real algebraic variety in the sense of Bochnak-Coste-Roy \cite{BCR}.
For a variety $X$ we denote by $C_*(X)$ the complex of semialgebraic chains of $\XR$ with coefficients in $\Z_2$ and closed supports.    
The homology of  $C_*(X)$ is the 
Borel-Moore homology of $\XR$  with $\Z_2$ coefficients,  and will be denoted  by $H_*(X)$.  

\subsection{Filtered complexes}\label{filteredcomplexes}

Let $\Cat$ be the category of bounded complexes of $\Z_2$  
vector spaces with  increasing bounded filtration,  
$$
\Kdot = \cdots \leftarrow K_0 \leftarrow K_1 \leftarrow K_2 \leftarrow  
\cdots, \qquad\cdots  \subset F_{p-1} \Kdot \subset F_p \Kdot \subset 
F_{p+1}\Kdot \subset \cdots. 
$$
Such a filtered complex defines a spectral
sequence $\{E^r,d^r\}$, $r=1,2,\dots$, with
$$
E^0_{p,q} = \frac {F_p K_{p+q}} {F_{p-1} K_{p+q}} , \qquad 
E^1_{p,q} = H_{p+q} \left(\frac {F_p \Kdot } {F_{p-1} \Kdot }\right),
$$
that converges to the homology of $\Kdot$,
$$
E^\infty_{p,q} = \frac {F_p(H_{p+q} \Kdot) } {F_{p-1} (H_{p+q}\Kdot) },
$$
where $F_p(H_n\Kdot) = \text{Image}[H_n(F_p\Kdot) \to H_n(\Kdot)]$ (\emph {cf.} 
\cite{maclane}, Thm. 3.1).
A \emph{quasi-isomor-phism} in $\Cat$ is a filtered quasi-isomorphism, \emph{i.\ e.} 
a  morphism of filtered complexes that induces an isomorphism on $E^1$.  Thus
a quasi-isomorphism induces an isomorphism of the associated spectral
sequences.

Following Guill\'en and Navarro Aznar (\cite{navarro}, (1.5.1)) we denote by $\HCat$ the category 
$\Cat$ localised with respect to filtered quasi-isomorphisms.   


Every bounded complex $\Kdot$ has a \emph{canonical filtration} 
\cite{deligne2} given by:
$$
F_{p}^{can} \Kdot = 
\begin{cases}
K_q \qquad & \text { if } q>-p\\
\ker \partial_q \qquad & \text { if } q=-p\\
0 \qquad & \text { if } q<-p\\
\end{cases}
$$
We have
\begin{equation}\label{canonical}
E^1_{p,q} = H_{p+q} \left(\frac {F^{can}_p \Kdot } {F^{can}_{p-1} \Kdot }\right)
= \begin{cases}
H_{p+q} (\Kdot ) \qquad & \text { if } p+q =-p\\
0 \qquad & \text {otherwise }\\
\end{cases}
\end{equation}
Thus a quasi-isomorphism of complexes induces a filtered quasi-isomorphism
of complexes with  canonical filtration.

To certain types of diagrams in $\Cat$ we can associate an element of $\Cat$,
the \emph{simple filtered complex} of the given diagram.
We use
notation from \cite{navarro}.  For $n\geq 0$ let $\square^+_n$ be the
partially ordered set of subsets of $\{0,1,\dots,n\}$. A
\emph{cubical diagram} of type $\square^+_n$ in a category $\mathcal X$ is a contravariant functor
from $\square^+_n$ to $\mathcal X$. If $\mathcal K$ is a cubical diagram in $\Cat$
of type $\square^+_n$, let $K_{*,S}$ be the complex labelled by the subset $S\subset \{0,1,\dots n\}$,
and let $|S|$ denote the number of elements of $S$. The simple complex $\s\mathcal K$ is defined by
$$
\s\mathcal K_k =  \bigoplus_{i+|S|-1=k}\mathcal K_{i,S}
$$
with differentials $\partial: \s \mathcal K_k \to \s \mathcal K_{k-1}$
defined as follows. For each $S$ let $\partial':K_{i,S}\to K_{i-1,S}$ be the differential of $K_{*,S}$.
If $T\subset S$ and $|T| = |S|-1$, let $\partial_{T,S}:K_{*,S}\to K_{*,T}$ be the chain map corresponding
to the inclusion of $T$ in $S$. If $a\in K_{i,S}$, let 
$$
\partial''(a) = \sum \partial_{T,S}(a),
$$
where the sum is over all $T\subset S$ such that $|T| = |S|-1$, and 
$$
\partial(a) = \partial'(a) + \partial''(a).
$$
The filtration of $\s \mathcal K$ is given by
$F_p \s \mathcal K = \s F_p \mathcal K$, 

$$
(F_p \s \mathcal K)_k  = \bigoplus_{i+|S|-1=k}F_p(\mathcal K_{i,S}).
$$ 


\subsection{The weight complex}\label{weightcomplex} 

To state the next theorem, we only need  to consider diagrams in of type $\square^+_0$
or type $\square^+_1$. The inclusion of a closed subvariety $Y\subset X$ is 
 a $\square^+_0$-diagram in $\Sch$.
An \emph{acyclic square} (\cite{navarro}, (2.1.1)) is a $\square^+_1$-diagram
in $\Sch$,
\begin{equation}\label{acyclic}\minCDarrowwidth 1pt\begin{CD}
\tilde Y @> >> \tilde X \\
 @V VV @VV\pi V \\
Y @>i >>X
\end{CD}\end{equation}
where $i$ is the inclusion of a closed subvariety, $\tilde Y = \pi^{-1}(Y)$, and the restriction
of $\pi$ is an isomorphism $\tilde X \setminus \tilde Y \to X \setminus Y$. An  \emph{elementary acyclic square}
 is an acylic square such that $X$ is compact and nonsingular, $Y$ is nonsingular, and $\pi$ is the blowup of $X$
along $Y$.


For a real algebraic variety $X$, let $F^{can} C_*(X)$ denote the complex $C_*(X)$
of semialgebraic chains with the canonical filtration.


\begin{thm} \label{navarro}
The functor 
$$
F^{can} C_*:\V \to \HCat
$$ 
that associates to a nonsingular projective variety  $M$ 
the semialgebraic chain complex of $M$
with  canonical filtration admits an extension to 
a functor defined for all real algebraic varieties and proper regular morphisms,
$$
\WC_*:\Sch \to \HCat,
$$ 
with the following properties:
\begin{enumerate}
\item 
Acyclicity. For an acyclic square (\ref{acyclic})
the simple filtered complex of the diagram
\begin{equation*}
\minCDarrowwidth 1pt\begin{CD}
\WC_{*}(\tilde Y) @> >> \WC_{*}(\tilde X) \\
 @V VV @VV  V \\
\WC_{*}(Y)@> >>\WC_{*}(X)
\end{CD}\end{equation*}
is acyclic (quasi-isomorphic 
to the zero complex). 
\item 
Additivity. For a closed inclusion $Y\subset X$,  
the simple filtered complex of the diagram
$$
 \WC_{*} (Y) \rightarrow \WC_{*} (X)
$$
is quasi-isomorphic to $\WC_{*} (X\setminus Y)$.  
\end{enumerate} 
Such a functor $\WC_*$ is unique up to a unique quasi-isomorphism. 
\end{thm}

\begin{proof} 
This theorem follows from \cite{navarro}, Theorem $(2.2.2)^{op}$. 
By Proposition $(1.7.5)^{op}$ of \cite{navarro}, the category $\Cat$, with the class of 
quasi-isomorphisms and  the operation of simple complex $\s$ defined above,
is a category of homological descent.  
Since it factors through $\Cat$, the functor $F^{can} C_*$ is $\mathbf\Phi$-rectified (\cite{navarro}, (1.6.5), (1.1.2)). 
Clearly $F^{can} C_*$   is additive for disjoint unions (condition (F1) of \cite{navarro}).
It remains  to check  condition (F2) for $F^{can} C_*$, that the simple
filtered complex associated to an elementary acyclic square is acyclic.  

Consider the elementary acyclic square (\ref{acyclic}).
Let $\mathcal K$ be the simple complex associated to the $\square^+_1$-diagram 
\begin{equation*}
\minCDarrowwidth 1pt\begin{CD}
F^{can} C_*(\tilde Y) @> >> F^{can} C_*(\tilde X) \\
 @V VV @VV  V \\
F^{can} C_*(Y)@> >>F^{can} C_*(X)
\end{CD}\end{equation*}
By definition of the canonical filtration, for each $p$ we have
$(F_p\s\mathcal K)_k/ (F_{p-1}\s\mathcal K)_k\neq 0$ only for $-p+2\geq k\geq -p-1$,
and the complex $(E^0_{p,*}, d^0)$ has the following form:
$$
0\rightarrow \frac{(F_p\s\mathcal K)_{-p+2}}{(F_{p-1}\s\mathcal K)_{-p+2}}\rightarrow
\frac{(F_p\s\mathcal K)_{-p+1}}{(F_{p-1}\s\mathcal K)_{-p+1}}\rightarrow
\frac{(F_p\s\mathcal K)_{-p}}{(F_{p-1}\s\mathcal K)_{-p}}\rightarrow
\frac{(F_p\s\mathcal K)_{-p-1}}{(F_{p-1}\s\mathcal K)_{-p-1}}\rightarrow 0.
$$
A  computation gives
\begin{eqnarray*} 
& &  H_{-p+2}(E^0_{p,*}) = 0 \\
& &   H_{-p+1}(E^0_{p,*}) = \text{Ker}[H_{-p}(\tilde Y) \to H_{-p}(Y)\oplus H_{-p}(\tilde X)]  \\  
& &  H_{-p}(E^0_{p,*}) = \text{Ker}[H_{-p}(Y)\oplus H_{-p}(\tilde X)\to H_{-p}(X)]/\text{Im}[H_{-p}(\tilde Y) \to H_{-p}(Y)\oplus H_{-p}(\tilde X)] \\  
& &    H_{-p-1}(E^0_{p,*}) = H_{-p}(X)/\text{Im}[H_{-p}(Y)\oplus H_{-p}(\tilde X)\to H_{-p}(X)].
\end{eqnarray*}
These groups are zero because for all $k$ we have the short exact sequence  of an elementary acyclic square,
\begin{equation}\label{short} 
0\rightarrow H_k(\tilde Y)\rightarrow H_k(Y)\oplus H_k(\tilde X)\rightarrow H_k(X)\rightarrow 0
\end{equation} 
(\emph{cf.} \cite{virtual}, proof of Proposition 2.1).
\end{proof} 

\rem\label{Fcan} The above argument shows that the functor $F^{can}$ is acyclic on any acyclic square \eqref{acyclic}, 
provided the varieties $X,Y,\tilde X, \tilde Y$ are nonsingular and compact. 

\vskip.1in
If $X$ is a real algebraic variety, the \emph{weight complex} of $X$ is the filtered complex $\mathcal WC_*(X)$.  
A stronger version of the uniqueness of $\WC_*$ is given by the following naturality theorem.

\begin{thm} \label{naturality}
Let $A_*$, $B_*:\V \to \Cat$ be functors such that their localizations
$\V \to \HCat$ satisfy the disjoint additivity condition (F1) and the elementary acyclicity
condition (F2). If $\tau:A_*\to B_*$ is a morphism of functors, then the localization
of $\tau$ extends uniquely to a morphism $\tau': \mathcal W A_*\to \mathcal W B_*$.
\end{thm}

\begin{proof}
This follows from $(2.1.5)^{op}$ and $(2.2.2)^{op}$ of \cite{navarro}.
\end{proof}

Thus if $\tau:A_*(M)\to B_*(M)$ is a quasi-isomorphism for all nonsingular projective varieties
$M$, then $\tau':\mathcal W A_*(X)\to \mathcal W B_*(X)$ is a quasi-isomorphism for 
all varieties $X$.

\begin{prop}\label{borelmoore}
For all real algebraic varieties $X$, the homology of the complex $\mathcal WC_*(X)$ is the
Borel-Moore homology of $X$ with $\Z_2$ coefficients,
$$
H_n(\mathcal WC_*(X)) = H_n(X).
$$
\end{prop} 
\begin{proof}
Let $\mathcal D$ be the category of bounded complexes of $\Z_2$ vector spaces.
The forgetful functor $\Cat\to \mathcal D$ 
induces a functor $\varphi: \HCat\to Ho\mathcal D$. To see this, let $A'_*$, $B'_*$ be filtered
complexes, and let $A_*=\varphi (A'_*)$ and $B_*=\varphi (B'_*)$. A quasi-isomorphism
$f:A'_*\to B'_*$ induces an isomorphism of the corresponding spectral sequences,
which implies that $f$ induces an isomorphism $H_*(A_*) \to H_*(B_*)$; in other words $f:A_*\to B_*$
is a quasi-isomorphism.

Let $C_*:\Sch\to Ho\mathcal D$ be the functor that assigns to every real
algebraic variety $X$ the complex of semialgebraic chains $C_*(X)$. Then $C_*$
satisfies properties (1) and (2) of Theorem \ref{navarro}. Acyclicity of $C_*$ 
for  an acyclic square (\ref{acyclic}) 
follows from the short exact sequence of chain complexes
$$
0\rightarrow C_*(\tilde Y)\rightarrow C_*(Y)\oplus C_*(\tilde X)\rightarrow C_*(X)\rightarrow 0.
$$
The exactness of this sequence follows immediately from the definition of semialgebraic
chains.
Similarly, additivity of $C_*$ for  a closed embedding $Y\to X$ follows from the short exact sequence 
of chain complexes
$$
0\rightarrow C_*(Y)\rightarrow C_*(X)\rightarrow C_*(X\setminus Y)\rightarrow 0.
$$

Now consider the functor $\mathcal WC_*:\Sch\to \HCat$ given by Theorem \ref{navarro}.
The functors $\varphi\circ \mathcal WC_*$ and $C_*:\Sch\to Ho\mathcal D$ are
extensions of $C_*:\V \to Ho\mathcal D$, so by \cite{navarro} Theorem
$(2.2.2)^{op}$ we have that $\varphi(\mathcal WC_*(X))$ is quasi-isomorphic to
$C_*(X)$ for all $X$. Thus $H_*(\mathcal WC_*(X))=H_*(X)$, as desired.
\end{proof}


\subsection{The weight spectral sequence}

If $X$ is a real algebraic variety, the  \emph{weight spectral sequence} of $X$, $\{E^r, d^r\}$, $r = 1, 2, \dots$,  is the spectral sequence of the weight complex $\mathcal WC_*(X)$. It is well-defined by Theorem \ref{navarro}, and it converges to the homology of $X$
by Proposition \ref{borelmoore}. The associated filtration of the homology of $X$ is the \emph{weight filtration}:
$$
0 = \mathcal W_{-k-1}H_k(X)\subset \mathcal W_{-k}H_k(X)\subset\cdots\subset \mathcal W_{0}H_k(X) = H_k(X),
$$
where $H_k(X)$ is the homology with closed supports (Borel-Moore homology) with
coefficients in $\Z_2$. (We will show $ \mathcal W_{-k-1}H_k(X)= 0$ in Corollary \ref{wss} below.) The dual weight filtration on cohomology with compact supports is discussed in \cite{virtual}.

\rem By analogy with Deligne's weight filtration on the cohomology of a complex algebraic
variety \cite{deligne}, there should also be a weight filtration on the homology of a real algebraic variety
with classical compact supports and coefficients in $\Z_2$ (dual to cohomology with closed supports). We plan to study this filtration in subsequent work.

\vskip.1in

The weight spectral sequence $E^r_{p,q}$ is a second
quadrant spectral sequence.  (We will show in Corollary \ref{wss} that if $E^1_{p,q}\neq 0$ then
$(p,q)$ lies in the closed triangle with vertices $(0,0)$, $(0,d)$,
$(-d,2d)$, where $d = \dim X$.)
The reindexing 
$$
p' = 2p+q, \ \  q' = - p,\ \  r' =  r + 1
$$
gives a standard first quadrant spectral sequence, with 
$$
\tilde E^2_{p',q'}  = E^ 1_{-q',p'+2q'}. 
$$
(If $\tilde E^2_{p',q'}\neq 0$ then
$(p',q')$ lies in the closed triangle with vertices $(0,0)$, $(d,0)$,
$(0,d)$, where $d = \dim X$.) Note that the total grading is preserved: $p' + q' = p+q$.

The virtual Betti numbers \cite{virtual} are the Euler characteristics of the rows of $\tilde E^2$:
\begin{equation}\label{virtualbetti}
\beta_q(X) = \sum_p(-1)^p\dim_{\Z_2}\tilde E^2_{p,q}
\end{equation}
To prove this assertion we will show that the numbers $\beta_q(X)$ defined by (\ref{virtualbetti}) are additive
and equal to the classical Betti numbers for $X$ compact and nonsingular.

For each $ q\geq 0$ consider the chain complex defined by the $ q$th row of the $\tilde E^1$ term,
$$
C_*(X,q) = (\tilde E^1_{*, q}, \tilde d^1_{*, q}),
$$
where $\tilde d^1_{ p, q}: \tilde E^1_{ p, q}\to \tilde E^1_{ p-1, q}$. This chain complex is well-defined up to quasi-isomorphism, and its Euler characteristic is $\beta_{ q}(X)$. 

The additivity of $\mathcal WC_*$ implies that if $Y$ is a closed subvariety of $X$ then the chain
complex
$C_*(X\setminus Y,  q)$ is quasi-isomorphic to the mapping cone of the chain map
$C_*(Y,  q) \to C_*(X,  q)$, and hence there is a long exact sequence of homology groups
$$
\cdots \to \tilde E^2_{ p, q}(Y)\to  \tilde E^2_{ p, q}(X)\to  \tilde E^2_{ p, q}(X\setminus Y)\to 
\tilde E^2_{ p-1, q}(Y)\cdots .
$$
Therefore for each $q$ we have
$$
\beta_{ q}(X) = \beta_{ q}(X\setminus Y) + \beta_{ q}(Y).
$$
This is the additivity property of the virtual Betti numbers.

\rem Navarro Aznar pointed out to us  that $C_*(X, q)$ is actually well-defined up to chain homotopy equivalence. One merely applies \cite{navarro}, Theorem $(2.2.2)^{op}$, to the functor that assigns to a nonsingular projective variety $M$ the chain complex
$$
C_k(M, q) = \begin{cases}
H_{ q}(M)\qquad k = 0\\
0\qquad \qquad k\neq 0\\
\end{cases}
$$
in the category of bounded complexes of $\Z_2$ vector spaces localized with respect to chain homotopy equivalences. This  striking application of the theorem of Guill\'en and Navarro Aznar led to our proof of the existence of the weight complex.

\vskip.1in

We say the weight complex  is \emph{pure} if the reindexed
weight spectral sequence has
$\tilde E^2_{ p,  q} =0$ for $p \neq0$.  In this case
the numbers $\beta_q(X)$ equal the classical Betti numbers of $X$.

\begin{prop}\label{purity}
If $X$ is a compact nonsingular variety, the weight complex
$\WC_{*} (X)$ is pure. In other words, if $k\neq - p$ then
$$
H_k\left( \frac{\mathcal W_pC_*(X)}{\mathcal W_{p-1}C_*(X)} \right)=0.
$$
\end{prop}

\begin{proof} For $X$  projective and nonsingular, the filtered complex $\WC_*(X)$ is quasi-isomorphic
to $C_*(X)$ with the canonical filtration.
The inclusion $\V \to \Reg$ has the extension property of \cite{navarro} (2.1.10); the proof is
similar to \cite{navarro} (2.1.11). Therefore by
 \cite{navarro} Theorem $(2.1.5)^{op}$,  the functor $
F^{can} C_*:\V \to \HCat$ extends to a functor  $\Reg \to \HCat $ that is additive for disjoint unions and acyclic, and this extension is unique up to  quasi-isomorphism.  But  $F^{can}C_* : \Reg \to \HCat $ is such an extension, since $F^{can}C_*$ is additive for disjoint unions in $\Reg$ and acyclic for acyclic squares in $\Reg$ (\emph{cf.}\ the proof of Theorem \ref{navarro} and Remark \ref{Fcan}).
\end{proof}

If $X$ is compact, we will show that the reindexed weight spectral sequence $\tilde E^{ r}_{ p, q}$
is isomorphic to the spectral sequence of a \emph{cubical hyperresolution} of $X$
\cite{navarro} (\emph{cf.} \cite{peterssteenbrink}, ch. 5). 

A cubical hyperresolution of $X$ is a special type of $\square^+_n$-diagram with
final object $X$ and all other objects compact and nonsingular. Removing $X$ gives
a $\square_n$-diagram, which is the same thing as a $\triangle_n$-diagram,
\emph{i.e.}\ a diagram labelled by the simplices contained in the standard
$n$-simplex $\triangle_n$. (Subsets of $\{0,1,\dots,n\}$ of cardinality $i+1$
correspond to $i$-simplices.)

The spectral sequence of a cubical hyperresolution is the spectral  sequence of
the filtered complex $(C_*, \hat F)$, with $C_k = \bigoplus_{i+j=k} C_jX^{(i)}$, where $X^{(i)}$ is the disjoint
union of the objects labelled by $i$-simplices of $\triangle_n$, and the filtration $\hat F$
is by skeletons, 
$$
\hat F_pC_k = 
\bigoplus_{i\leq p}  C_{k-i}X^{(i)}
$$
The resulting first quadrant spectral sequence $\hat E^r_{p,q}$ 
converges to the homology of $X$, and the associated filtration is the weight
filtration as defined by Totaro \cite{totaro}.

Let $\partial = \partial' + \partial''$ be the boundary operator of the complex $C_*$, where
$\partial'_i: C_jX^{(i)}\to C_jX^{(i-1)}$ is the simplicial boundary operator, and 
$\partial''_j:C_jX^{(i)}\to C_{j-1}X^{(i)}$ is $(-1)^i$ times the boundary operator on semialgebraic chains.

\begin{prop}\label{cubical spectral sequence}
If $X$ is a compact variety, the weight spectral sequence $E$ of $X$ is isomorphic
to the spectral sequence $\hat E$ of a cubical hyperresolution of $X$,
$$E^r_{p,q}\cong \hat E^{r+1}_{2p+q, -p}.$$
Thus $\hat E^r_{p,q} \cong \tilde E^r_{p,q}$, the reindexed weight spectral sequence introduced above.
\end{prop}

\begin{proof}

The acyclicity property of the weight complex ((1) of Theorem \ref {navarro}) implies
that $\WC_*$ is acyclic for cubical hyperresolutions (see \cite{navarro}, proof of Theorem
(2.1.5)). In other words, if the functor
$WC_*$ is applied to a cubical hyperresolution of $X$, the resulting $\square^+_n$-diagram
in $\Cat$ is acyclic. This says that
 $\WC_*(X)$ is filtered quasi-isomorphic
to the total filtered complex of the double complex $\WC_{i,j} = \WC_jX^{(i)}$. Since the
varieties $X^{(i)}$ are compact and nonsingular, this filtered complex is quasi-isomorphic to the total
complex $C_k = \bigoplus_{i+j=k}C_jX^{(i)}$ with the canonical filtration,
$$
F^{can}_pC_k = \Ker \partial''_{-p} \oplus \bigoplus_{j>-p}C_jX^{(k-j)}.
$$
Thus the spectral sequence of this filtered complex  is the weight spectral sequence $E^r_{p,q}$.

We now compare the two increasing filtrations $F^{can}$ and $\hat F$ on the complex $C_*$. The weight spectral sequence $E$ is associated to the filtration $F^{can}$, and the cubical hyperresolution spectral sequence $\hat E$ is associated to the filtration $\hat F$. We show that $F^{can} = \text{Dec}\,(\hat F)$,  the \emph{Deligne shift} of $\hat F$ (\emph{cf.} \cite{deligne2} (1.3.3), \cite{peterssteenbrink} A.49).

Let $\hat F'$ be the filtration
$$
\hat F'_pC_k = \hat Z^1_{p,k-p} =  \Ker[\partial: \hat F_pC_k\to C_{k-1}/\hat F_{p-1}C_{k-1}],
$$
and let $\hat E'$ be the associated spectral sequence. By definition of the Deligne shift,
$$
\hat F'_pC_k = \text{Dec}\, \hat F_{p-k}C_k.
$$
Now since $\partial = \partial' + \partial''$ it follows that
$$
\hat F'_pC_k = F^{can}_{p-k}C_k,
$$ 
and $F^{can}_{p-k}C_k =  F^{can}_{-q}C_k$, where $p +q = k$. Thus we can identify the spectral sequences
$$
(\hat E')^{r+1}_{p,q} = E^r_{-q,p+2q} \ \ (r\geq 1).
$$
On the other hand, the inclusion $\hat F'_pC_k \to \hat F_pC_k$ induces an isomorphism of spectral sequences 
$$
(\hat E')^r_{p,q}\cong \hat E^r_{p,q} \ \ (r\geq 2).
$$
\end{proof}

\example If $X$ is a compact divisor with normal crossings in a nonsingular variety, a cubical hyperresolution
of $X$ is given by the decomposition of $X$ into irreducible components.  (The corresponding simplicial diagram
 associates to an $i$-simplex the disjoint union of the intersections of
$i+1$ distinct irreducible components of $X$.)
The spectral sequence of such a cubical hyperresolution is 
the Mayer-Vietoris (or \u Cech) spectral sequence associated to the decomposition. Example 3.3 of \cite{virtual} is
an algebraic surface $X$ in affine 3-space such that $X$ is the union of three compact nonsingular surfaces with normal crossings
and the weight spectral sequence of $X$ does not collapse, \emph{i.e.}\ $\tilde E^2 \neq \tilde E^\infty$.
The variety $U=\R^3\setminus X$ is an example of a nonsingular noncompact variety with non-collapsing weight spectral
sequence. (The additivity property (2) of Theorem \ref{navarro} can be used to compute the spectral sequence of 
$U$.)

\begin{cor}\label{wss} Let $X$ be a real algebraic variety of dimension $d$, with weight spectral sequence $E$ and weight filtration $\mathcal W$. For all $p, q, r$, if $E^r_{p,q}\neq 0$ then $p\leq 0$ and $-2p\leq q \leq d-p$. Thus for all $k$ we have $\mathcal W_{-k-1}H_k(X) = 0$.
\end{cor}

\begin{proof} For $X$ compact this follows from Proposition \ref{cubical spectral sequence} and the fact that $\hat E^r_{p,q} \neq 0$ implies $p\geq 0$ and $0\leq q \leq d-p$. If $U$ is a noncompact variety, let $X$ be a real algebraic compactification of $U$, and let $Y= X\setminus U$. We can assume that $\dim Y < d$. The corollary now follows from the additivity property of the weight complex (Theorem \ref{navarro} (2)).
\end{proof}



\medskip
\section{A geometric filtration}
\label{geometric}

We define a functor 
$$
\A C_* : \Sch\to \Cat 
$$ 
that assigns to each real algebraic variety $X$  the complex $C_*(X)$ of semialgebraic chains 
of $X$ (with coefficients in $\Z_2$ and closed supports), together with a filtration 
\begin{equation} \label{geometricfiltration}
0 = \A_{-k-1 } C_k (X) \subset  \A_{-k} C_k (X) \subset \A_{-k+1}  
C_k (X) \subset \cdots  \subset \A_{0} C_k (X) = C_k (X) .
\end{equation} 
We  prove  in Theorem \ref{geomiso} that the functor $\A C_*$ realizes the  weight complex functor 
$\mathcal W C_* : \Sch \to \HCat$ given 
by Theorem \ref{navarro}.  Thus  the filtration $\A_*$ of chains
gives the weight filtration of homology.


\subsection{Definition of the filtration $\A_*$}\label{geometricdefinition}

The filtration will  first be defined for compact varieties.  Recall that $\XR$ denotes the set of
real points of the real algebraic variety $X$.

\begin{thm}\label{axioms}
  There exists a unique filtration \eqref{geometricfiltration} on semialgebraic $\Z_2$-chains 
of compact real algebraic varieties with the following properties. 
Let $X$ be a compact real algebraic variety and let $c\in C_k (X)$. Then 
\begin{enumerate}
\item
If  $Y\subset X$ is an algebraic subvariety such that $\supp c \subset \YR$ then 
$$
c\in \A_p C_k(X) \Longleftrightarrow c\in \A_p C_k(Y).
$$
\item
Let $\dim X =k$ and let $\pi : \tilde X \to X$ be a resolution of $X$  such that 
there is a normal crossing divisor $D\subset\tilde X$ with 
$\supp \partial (\pi \inv c) \subset \DR$.  Then for $p\ge -k$,  
$$
c\in \A_p C_k(X) \Longleftrightarrow \partial (\pi \inv c) \in \A_p C_{k-1}(D).
$$
\end{enumerate}
\end{thm}

\medskip
\noindent
We call a resolution $\pi : \tilde X \to X$ \emph{adapted} 
to $c\in C_k (X)$ if it satisfies condition (2) above.  For the definition of the support $\supp c$
and the
pullback $\pi^{-1}c$ see the Appendix.

\begin{proof}
We proceed by induction on $k$.  If $k=0$ then 
$0 =\A_{-1} C_0 (X) \subset \A_{0} C_0 (X) = C_0 (X)$.  In the rest of this subsection 
we assume the existence and uniqueness of the filtration for chains of dimension $<k$, and  
we prove the statement for  chains of dimension $k$.

\begin{lem}\label{components} 
Let $X= \bigcup_{i=1}^s X_i$ where $X_i$ are subvarieties of $X$. 
Then for $m<k$, 
$$ c\in \A_p C_m(X) \Longleftrightarrow \forall_i\  c|_{X_i} \in \A_p C_m(X_i) $$
\end{lem}

\begin{proof}
By (1) we may assume that $\dim X=m$ and then that all $X_i$ are distinct of dimension $m$.  Thus an adapted 
resolution of $X$ is a collection of adapted resolutions of each component of $X$.  
\end{proof}

See the Appendix for the definition of the restriction $c|_{X_i}$.

\begin{prop}\label{functoriality}
The filtration $\A_p$ given by Theorem \ref{axioms} is functorial; that is, for a regular 
morphism $f:X\to Y$ of compact real algebraic varieties, 
$f_* (\A_p C_m (X)) \subset \A_p C_m(Y)$, for $m<k$.
\end{prop} 

\begin{proof}
We prove that if the filtration satisfies the statement of Theorem \ref{axioms} 
for chains of dimension $<k$ and is functorial on chains of dimension $<k-1$ 
then it is functorial on chains of dimension $k-1$.  

Let $c\in C_{k-1} (X)$, and let $f:X\to Y$ be a regular morphism of compact 
real algebraic varieties.  By (1) of Theorem \ref{axioms} we may assume 
$\dim X = \dim Y = k-1$ and by Lemma \ref{components} that $X$ and $Y$ are irreducible.  
We may assume that $f$ is dominant; otherwise $f_* c =0$.  
 Then there exists a commutative diagram 
 \begin{equation*}\minCDarrowwidth 1pt\begin{CD}
\tilde X @>\tilde f >> \tilde Y \\
 @V\pi_X VV @VV\pi_YV \\
X @>f>> Y 
\end{CD}\end{equation*}  
where $\pi_X$ is a resolution of $X$ adapted to $c$ and $\pi_Y$ a resolution 
of $Y$ adapted to $f_* c$.   Then 
\begin{eqnarray*}
& c \in \A_p (X) \Leftrightarrow \partial (\pi_X \inv c) \in \A_p (\tilde X) \Rightarrow 
 \tilde f_* \, \partial (\pi_X \inv c)\in \A_p (\tilde Y), \\
&  \tilde f_* \, \partial (\pi_X \inv c) = \partial \tilde f_* (\pi_X
 \inv c)= \partial (\pi_Y \inv f_* c), \\
& \partial (\pi_Y \inv f_* c) \in \A_p (\tilde Y)
\Leftrightarrow f_*c\in \A_p (Y) , 
\end{eqnarray*}
where the implication in the first line follows from the inductive assumption.  
\end{proof}

\begin{cor}\label{geomboundary}
The boundary operator $\partial$ preserves the filtration $\A_p$, 
$$
\partial \A_p C_m (X)) \subset \A_p C_{m-1}(X),
$$ 
for $m < k$.
\end{cor} 

\begin{proof}
Let $\pi : \tilde X\to X$ be a resolution of $X$ adapted to $c$.  Let 
$\tilde c= \pi \inv c$.  Then $c= \pi_* \tilde c$ and   
$$
c\in \A_p \Leftrightarrow \partial \tilde c \in \A_p \Rightarrow 
\partial c = \partial \pi_* \tilde c = \pi_* \partial \tilde c   \in \A_p .  
$$
\end{proof}

Let $c\in C_k (X)$, $\dim X=k$.  In order to show  that  
the condition (2) of Theorem \ref{axioms} is independent of the choice 
of $\tilde \pi$   we need the following  lemma.

\begin{lem}\label{blowing-up}
Let $X$ be a nonsingular compact real algebraic variety of dimension $k$ and 
let $D\subset X$ be a normal crossing divisor.  Let $c\in C_k (X)$ 
satisfy $\supp \partial c \subset D$.  Let $\pi : \tilde X \to X$ be the blowup 
of a nonsingular 
subvariety $C\subset X$ that has normal crossings with $D$.  Then 
$$
\partial c \in \A_p C_{k-1} (X) \Longleftrightarrow \partial (\pi \inv (c)) \in \A_p C_{k-1} (\tilde X).
$$
\end{lem}

\begin{proof}
Let $\tilde D = \pi \inv (D)$.  Then $\tilde D= E\cup \bigcup \tilde D_i$, where 
$E=\pi \inv (C)$ is the exceptional divisor and $\tilde D_i$ denotes the strict 
transform of $D_i$.  By Lemma \ref{components}   
$$
\partial c \in \A_p C_{k-1} (X) \Longleftrightarrow \forall_i \,  \partial c|_{D_i} \in \A_p C_{k-1} (D_i). 
$$
Let $\partial_i c = \partial c|_{D_i}$.  The restriction $\pi_i=\pi|_{\tilde D_i} : \tilde D_i \to D_i$ is the blowup 
with smooth center $C\cap D_i$.  Hence, by inductive assumption, 
$$
\partial (\partial_i c )  \in \A_p C_{k-2} (D_i) \Longleftrightarrow  
\partial \pi _i \inv (\partial_i c ) =
\partial (\partial (\pi \inv (c))|_{\tilde D_i})  \in \A_p C_{k-2} ( \tilde D_i)
$$
By the inductive assumption of Theorem \ref{axioms}, 
$$
\partial (\partial_i c)  \in \A_p C_{k-2} (D_i)  \Longleftrightarrow \partial_i c  \in \A_p C_{k-1} (D_i) ,  
$$
and we have similar properties for $\partial (\pi \inv (c))|_{\tilde D_i}$ and $\partial (\pi \inv (c))|_{E}$.  

Thus to complete the proof it suffices to show that 
$$
\forall_i \, \partial (\partial (\pi \inv (c)|_{\tilde D_i} ) \in \A_p C_{k-2} (\tilde D_i) \Longrightarrow 
\partial (\partial (\pi \inv (c)|_{E} )\in \A_p C_{k-2} (E).
$$
This follows from $0= \partial (\partial \pi \inv (c)) = 
\partial \left(\sum_i \partial (\pi \inv (c)|_{\tilde D_i} + 
\partial (\pi \inv (c)|_{E}\right)$.
\end{proof}

\medskip
Let $\pi_i: X_i \to X,\ i=1,2,$ be two resolutions of $X$ adapted to $c$.  
Then there exists $\sigma: \tilde X_1 \to X_1$, the composition of finitely 
many blowups with smooth centers that have normal crossings with the strict 
transforms of all exceptional divisors, such that $\pi_1\circ \sigma$ factors through $X_2$,
\begin{equation*}\minCDarrowwidth 1pt\begin{CD}
\tilde X_1 @>\sigma>>  X_1 \\
 @V\rho VV @VV\pi_1V \\
X_2 @>\pi_2>> X
\end{CD}\end{equation*}  
By Lemma \ref{blowing-up}
$$
\partial (\pi_1 \inv (c)) \in \A_p C_{k-1} (X_1) \Leftrightarrow 
\partial (\sigma \inv (\pi_1 \inv (c))) \in \A_p C_{k-1} (\tilde X_1)  .
$$
On the other hand,
$$
\rho_* \partial (\sigma \inv (\pi_1 \inv (c))) 
= \rho_* \partial (\rho \inv (\pi_2 \inv (c))) 
= \partial (\pi_2 \inv (c)) , 
$$
and consequently by Proposition \ref{functoriality} we have
$$
\partial (\pi_1 \inv (c)) \in \A_p C_{k-1} (X_1) 
 \Longrightarrow \partial (\pi_2 \inv (c)) \in \A_p   C_{k-1} ( X_2).
$$
By symmetry 
$\partial (\pi_2 \inv (c)) \in \A_p(X)  \Rightarrow \partial (\pi_1 \inv (c)) \in \A_p(X) $. 
This completes the proof of Theorem \ref{axioms}.
\end{proof}
\medskip

\subsection{Properties of the filtration $\A_*$}

Let $U$ be a (not necessarily compact) real algebraic variety and let $X$ 
be a real algebraic compactification of $U$.   We extend the 
filtration $\A_p$ to $U$ as follows.   
If  $c\in C_*(U)$ let $\bar c \in C_*(X)$ be its closure. We define 
$$
c\in \A_p  C_k (U) \Leftrightarrow \bar c \in \A_p C_k(X).
$$
See the Appendix for the definition of the closure of a chain. 

\begin{prop}\label{Gcomp}
$\A_p  C_k (U)$ is well-defined; that is, for two compactifications $X_1$ and 
$X_2$ of $U$, 
$$
c_1\in \A_p  C_k (X_1) \Leftrightarrow c_2 \in \A_p C_k(X_2) ,
$$
where $c_i$ denotes the closure of $c$ in $X_i, i=1,2$.    
\end{prop}

\begin{proof}
We may  assume that $k=\dim U$. 
By a standard argument any two compactifications can be dominated by a third one.  Indeed, 
denote the inclusions by $i_i : U \hookrightarrow X_i$.  Then 
the Zariski closure $X$ of the image of $(i_1,i_2)$ in $X_1\times X_2$ is a compactification 
of $U$.  

Thus we may assume that there is a morphism $f:  X_2\to X_1$ that is the identity on $U$.  
 Then, 
by functoriality, 
$$
c_2\in \A_p  C_k (X_2) \Rightarrow c_1=f_*(c_2) \in \A_p C_k(X_1) .
$$ 
By the Chow-Hironaka Lemma there is a resolution $\pi_1: \tilde X_1\to X_1$, adapted to $c_1$,  that factors through $f$; that is,
$\pi_1= f\circ g$.  Then 
$$
c_1\in \A_p  C_k (X_1) \Leftrightarrow  \pi_1 \inv (c_1)\in \A_p  C_k (\tilde X_1) 
\Longrightarrow c_2=g_*( \pi_1 \inv (c_1)) \in \A_p C_k(X_2) .
$$ 
\end{proof}

\begin{thm} \label{geomproperties} 
The filtration $\A_*$ defines a functor  $\A C_* :\Sch \to \Cat$ with the following properties:
\begin{enumerate}
\item 
For an acyclic square (\ref{acyclic})
the following  sequences are exact:
\begin{eqnarray*} \label {acexact1}
& &  0\to \A_p  C_{k} (\tilde Y) \to \A_p C_k (Y) \oplus \A_p  C_{k} (\tilde X) \to 
\A_p C_{k} (X) \to 0 \\ \label {acexact2}
& &  0\to \frac{\A_pC_{k} (\tilde Y)}{\A_{p-1}  C_{k} (\tilde Y)} \to 
\frac{\A_pC_k (Y)}{\A_{p-1} C_k (Y)} \oplus \frac {\A_p C_{k} (\tilde X)}{\A_{p-1}  C_{k} (\tilde X)} \to 
\frac{\A_pC_{k} (X)}{\A_{p-1} C_{k} (X)} \to 0  . 
\end{eqnarray*} 
\item 
 For a closed inclusion $Y\subset X$,  with $U=X\setminus Y$, the following sequences are exact:
\begin{eqnarray*} \label {exact1}
& &  0\to \A_p  C_{k} (Y) \to \A_p  C_{k} (X) \to 
\A_p C_{k} (U) \to 0 \\ \label {exact2}
& & 0\to \frac{\A_pC_{k} (Y)}{\A_{p-1}  C_{k} (Y) }\to \frac{\A_p C_{k} (X)}{\A_{p-1}  C_{k} (X)} \to 
\frac{\A_pC_{k} (U) }{\A_{p-1} C_{k} (U)} \to 0  . 
\end{eqnarray*}
\end{enumerate}   
\end{thm}

\begin{proof}
The exactness of the first sequence of (2) follows directly from the definitions  
(moreover, this sequence splits via $c\mapsto \bar c$).  The exactness of the second sequence of (2) 
now follows by a diagram chase. Similarly, the exactness of the first
sequence of (1) follows from the definitions, and the exactness of the second sequence of (1)
is proved by a diagram chase.
\end{proof}

%

For any variety  $X$, the filtration  $\A_*$ is contained in the canonical filtration,
\begin{equation}\label{geomcan}
\A_p C_k (X) \subset F^{can}_p C_k (X),
\end{equation}
since $\partial_k(\A_{-k}C_k(X))= 0$.
Thus on the category of nonsingular projective varieties we have a morphism of functors
$$
\sigma: \A C_*  \to F^{can} C_*.
$$
\begin{thm}\label{geomiso}
For every nonsingular projective real algebraic variety $M$,   
$$
\sigma (M): \A C_*(M)  \to F^{can} C_*(M)
$$ 
is a filtered quasi-isomorphism.  Consequently, for every real algebraic variety $X$ 
the localization of $\sigma$ induces a quasi-isomorphism 
$\sigma' (X): \A C_* (X) \to \mathcal WC_*(X)$.   
\end{thm}

Theorem \ref{geomiso} follows from Corollary \ref{iso} and Corollary \ref{2filtrations}, which 
will be shown  in the next section.



\section{The Nash constructible filtration}\label{nash}
\medskip

In this section we  introduce the \emph{Nash 
constructible filtration} 
\begin{equation}\label{Nashfiltration}
0 = \Na_{-k-1 } C_k (X) \subset  \Na_{-k} C_k (X) \subset \Na_{-k+1}  
C_k (X) \subset \cdots  \subset \Na_{0} C_k (X) = C_k (X)
\end{equation} 
on the semialgebraic chain complex $C_*(X)$ of a real algebraic variety $X$.  We show that this filtration induces a functor 
$$
\Na C_* : \Sch \to \Cat 
$$ 
that realizes the  weight complex functor $\mathcal W C_* : \Sch \to \HCat$.  In order to prove this assertion in
Theorem \ref{iso}, we have to extend $\Na C_*$ to a wider category of sets and morphisms.  The objects of this
category are certain semialgebraic subsets of the set of real points of a real
algebraic variety, and they include in particular all connected  
components of real algebraic subsets of $\R \proj ^n$. The morphisms are certain proper
continuous semialgebraic maps between these sets. This extension 
is crucial for the proof.    
As a corollary we show that for  real algebraic varieties the Nash constructible filtration $\Na_*$ coincides 
with the geometric filtration $\A_*$ of Section \ref{geometricfiltration}.  In this way we complete the proof of Theorem \ref{geomiso}.

For real algebraic varieties, the Nash constructible filtration was first defined in an unpublished paper of H. Pennaneac'h \cite{penn2}, 
by analogy with the algebraically constructible filtration (\cite{penn1}, \cite{penn3}).  Theorem \ref{iso} gives, in particular, that the Nash constructible filtration of a
compact variety is the same as 
the filtration  given by a cubical hyperresolution; this answers affirmatively a question of Pennaneac'h  
(\cite{penn2} (2.9)). 

 

\smallskip
\subsection{Nash constructible functions on $\R \proj ^n$ and arc-symmetric sets}

In real algebraic geometry it is common to work with real algebraic 
subsets of the affine space $\R^n\subset \R \proj ^n$ instead of schemes over $\R$, and with 
(entire) regular  rational 
mappings as morphisms; see for instance \cite{akbulutking} or \cite{BCR}.  Since 
$\R \proj ^n$ can be embedded in $\R^N$ by a biregular rational map (\cite{akbulutking}, 
\cite{BCR} (3.4.4)), this category also contains algebraic subsets of $\R \proj ^n$.  

  A \emph{Nash constructible function} on 
$\R \proj ^n$ is an integer-valued function $\varphi :\R \proj ^n \to \Z$ such that there exist a finite family of  regular rational mappings  $f_i:Z_i\to \R \proj ^n$ defined on projective real algebraic sets $Z_i$, connected components $Z'_i$ of $Z_i$,
and integers $m_i$,  such that for all $x\in \R \proj ^n$, 
\begin{equation} \label{Nash}
\varphi (x) = \sum_i m_i \chi ( 
f_{i}\inv (x)\cap Z'_i),
\end{equation}
where $\chi$ is the Euler characteristic.
Nash constructible functions were introduced in
\cite{mccroryparusinski}.  
  Nash constructible functions on $ \R \proj ^n$ form a ring.

  \begin{example}\label{ASsets}  \hfil
 \begin{enumerate}
  \item
 If $Y\subset \R \proj ^n$ is Zariski constructible (a finite set-theoretic combination of algebraic subsets),  then its characteristic function $\1_Y$ is Nash 
 constructible.  
  \item
   A subset $S\subset \R \proj ^n$   is called \emph{arc-symmetric} if 
every real analytic arc $\gamma : (a,b)\to  \R \proj ^n$ either meets  $S$ at isolated points  or  is entirely
included in $S$.   
Arc-symmetric sets were first studied  by  K. Kurdyka in \cite 
{kurdyka1}. 
As shown in \cite{mccroryparusinski}, a semialgebraic set $S\subset \R \proj ^n$ 
is arc-symmetric if and only if it is closed in $\R \proj ^n$ and $\1_S$ is Nash constructible. 
By the existence of arc-symmetric closure (\emph{cf.} \cite 
{kurdyka1}, \cite{aussois}), for a set $S\subset \R \proj ^n$ the function $\1_S$ is Nash constructible and only if 
$S$ is   a finite set-theoretic combination of semialgebraic arc-symmetric subsets of  $\R \proj ^n$.   If $\1_S$ is Nash constructible we say that $S$ is an \emph{$\AS $ set}.
\item 
  A connected component of a compact 
algebraic subset of $\R \proj ^n$ is  arc-symmetric.  A compact real analytic and semialgebraic subset 
of $\R \proj ^n$ is  arc-symmetric.   
\item
Every Nash constructible function on $\R \proj ^n$ is in particular \emph{constructible} (constant on strata of a 
finite semialgebraic stratification of $\R \proj ^n$).  Not all constructible functions are Nash constructible.  
By \cite{mccroryparusinski},  every  constructible function
 $\varphi : \R \proj ^n  \to 2^{n}\Z$ is Nash constructible.  
\end{enumerate}
 \end{example}

Nash constructible functions form the smallest family of constructible functions that contains 
characteristic functions of connected components of compact real algebraic sets, and that is stable under
 the natural 
 operations
inherited from sheaf theory:  pullback by regular rational morphisms, 
pushforward  by proper regular rational morphisms, restriction to Zariski open sets, and duality; see 
\cite{mccroryparusinski}.  In terms of the \emph{pushforward} (fiberwise integration with respect to the 
Euler characteristic)  
the formula  \eqref{Nash} can be expressed 
as $ \varphi =  \sum_i m_i \, f_{i\, *} \1_{Z'_i}$.  
Duality is closely related
to the \emph{link operator}, an important tool for studying the topological properties of real algebraic sets.   
   For more on Nash constructible function see \cite{bonnard} and \cite{aussois}.  
 
If $S\subset \R \proj ^n$ is an $\AS$ set (\emph{i.e.} $\1_S$ is Nash constructible), we say that
 a function on $S$ is \emph{Nash constructible}
if it is the restriction of a Nash constructible function on $ \R \proj ^n$.  In particular, this defines Nash constructible functions on affine real algebraic sets. (In the non-compact case  this definition is more restrictive than 
 that of \cite{mccroryparusinski}.)


\smallskip
\subsection{Nash constructible functions on real algebraic varieties}

Let  $X$ be a real algebraic variety  and let $\XR$ denote the set of real points on $X$.   
We call a function $\varphi : \XR  \to \Z$
 \emph{Nash constructible} if its restriction to every affine chart is Nash constructible.  The following lemma shows that
 this extends our definition of Nash constructible functions on affine real algebraic sets.  

\begin{lem} 
If $X_1$ and $X_2$ are two projective compactifications of the affine real algebraic variety $U$, then 
$\varphi : \UR \to \Z$ is the restriction of a Nash 
constructible function on $\XR_1$ if and only if $\varphi$ is the restriction of a Nash constructible function on 
$\XR_2$. 
\end{lem}

\begin{proof}
We may suppose that there is a regular projective morphism $f:X_1 \to X_2$ that is an isomorphism on $U$;
\emph{cf.}\ the proof of Proposition \ref{Gcomp}. 
Then the statement follows from the following two properties of Nash constructible functions.  
If  $\varphi_2: \XR_2\to \Z$ is  Nash constuctible  then so is its pullback 
$f^* \varphi_2 = \varphi_2 \circ f : \XR_1 \to \Z$.  
If  $\varphi_1: \XR_1\to \Z$ is  Nash constuctible  then so is its pushforward 
$f_* \varphi_1  : \XR_2 \to \Z$.  
\end{proof}

 \begin{prop}\label{Nashadditivity}
Let $X$ be a real algebraic variety and let  $Y\subset X$ be a closed 
subvariety.  Let $U=X\setminus Y$.  Then $\varphi :\XR \to \Z$ is Nash constructible if and only if  
the restrictions of $\varphi$ to $\YR$ and $\UR$ are Nash constructible.
\end{prop}

\begin{proof}
It suffices to check the assertion for $X$ affine; this case is easy.  
\end{proof}

\begin{thm}
Let $X$ be a complete real algebraic variety.  The function $\varphi:\XR \to \Z$ is Nash constructible if and only if 
 there exist a finite family of regular morphisms  $f_i:Z_i\to X$ defined on complete real algebraic 
 varieties $Z_i$, connected components $Z'_i$ of $\ZR_i$,
and integers $m_i$,  such that for all $x\in \XR$, 
\begin{align} \label{Nashsecond}
\varphi =  \sum_i m_i \, f_{i\, *} \1_{Z'_i}. 
\end{align}
\end{thm}

\begin{proof}  
 If $X$ is complete but not projective, then $X$ can be dominated by a birational regular morphism 
$\pi: \tilde X \to X$, with $\tilde X$ projective (Chow's Lemma).   Let $Y\subset X$, $\dim Y < \dim X$, be a closed 
subvariety such that $\pi$ induces an isomorphism   $\tilde X \setminus \pi\inv (Y) \to X \setminus Y$.  
Then, by Proposition 
\ref{Nashadditivity},  $\varphi :\XR \to \Z$ is Nash constructible if and only if $\pi^* \varphi$ and $\varphi$ 
restricted to $\YR$ are Nash constructible.  

Let $Z$ be a complete real algebraic variety and let  $f:Z\to X$ be a regular morphism.  
 Let $Z'$ be a connected component 
of $\ZR$.  We show that $\varphi =  f_*\1_{Z'}$ is Nash constructible.  This is obvious if both $X$ and $Z$ 
are projective.  If they are not, we may dominate both $X$ and $Z$ by projective varieties, using Chow's Lemma, 
and reduce to the projective case by induction on dimension.  

Let $\varphi:\XR \to \Z$ be Nash constructible.  Suppose first that $X$ is projective.  Then 
$\XR\subset \R \proj ^n$ is a real algebraic set.  Let $A\subset  \R \proj ^m$ be a real algebraic set and let 
$f:A\to \XR$ be a regular rational morphism $f=g/h$, where $h$ does not vanish on $A$, cf. 
\cite{akbulutking}.  Then the graph 
of $f$ is an algebraic subset $\Gamma\subset  \R \proj ^n\times  \R \proj ^m$ and the 
set of real points of a projective real variety $Z$.  Let $A'$ be a connected component of $A$, 
and $\Gamma'$ the 
graph of $f$ restricted to $A'$.  
Then $f_* \1_{A'} = \pi_* \1_{\Gamma'}$, where $\pi$ denotes the projection on the second factor.     

If $X$ is complete but not projective, we again dominate it by a birational regular morphism $\pi: \tilde X \to X$, 
with $\tilde X$ projective.  Let $\varphi:\XR \to \Z$ be Nash constructible.  Then 
$\tilde\varphi = \varphi \circ \pi : \tilde \XR \to \Z$ is Nash constructible.  Thus, by the case considered above, there are 
regular morphisms  $\tilde f_i: \tilde Z_i\to \tilde X$, and connected components $\tilde Z'_i$ such that 
\begin{align*}
\tilde \varphi (x) = \sum_i m_i \, \tilde f_{i\, *} \1_{\tilde Z'_i}.  
\end{align*}
Then $\pi_* \tilde \varphi = \sum_i m_i \, \tilde \pi_* f_{i\, *} \1_{\tilde Z'_i}$ and differs from $\varphi$ only on the 
set of real points of a variety of dimension smaller than $\dim X$.  We complete the argument 
by induction on dimension.    
\end{proof}

If $X$ is a real algebraic variety, we again
say that $S\subset \XR$ is an \emph{$\AS$ set} if $\1_S$ is Nash constructible, 
and  $\varphi:S\to \Z$ is \emph{Nash constructible} if 
the extension of $\varphi$ to $\XR$ by zero is a Nash constructible function on $\XR$.  

\begin{cor}
Let $X, Y$ be complete real algebraic varieties and let $S$ be an $\AS$ subset 
of $\XR$,  and $T$ an $\AS$ subset of $\YR$.    Let $\varphi:S\to \Z$ and $\psi: T \to \Z$ be Nash constructible.  
Let $f:S \to T$ be a map with $\AS $ graph  $\Gamma\subset \XR \times \YR$ and  
let $\pi_X : X\times Y \to X$ and  $\pi_Y : X\times Y\to Y$ denote the standard projections.  Then 
\begin{eqnarray}\label{Ngraphs1}
 & & f_* (\varphi) = (\pi_Y)_*( \1_\Gamma\cdot  \pi_X^* \varphi) \\ \label{Ngraphs2}
& & f^* (\psi) = (\pi_X)_*( \1_\Gamma \cdot \pi_Y^* \psi) 
\end{eqnarray}
are Nash constructible.  
\end{cor}


\medskip
\subsection{Definition of the Nash constructible filtration.}\label{defNashfiltration}
Denote by   $\mathcal X_{\AS}$ 
the category  of locally compact $\AS$ subsets of real algebraic varieties  as objects  and continuous proper 
maps with $\AS$ graphs as morphisms. 

 Let $T\in \mathcal X_{\AS}$. We say
that $\varphi: T\to \Z$ is \emph{generically Nash constructible  on 
$T$ in dimension $k$} if $\varphi$ coincides with a Nash constructible function
everywhere on $T$
 except on a semialgebraic subset of $T$ of dimension $<k$.  We say
 that  $\varphi$ is \emph{generically Nash constructible  on
$T$} if $\varphi$ is Nash constructible in dimension $d=\dim T$.

Let $c\in C_k (T)$, and let $-k\le p\leq 0$.  We say  that  $c$ is \emph{$p$-Nash constructible}, 
and write $c\in \Na_p C_k (T)$, if there exists  $\varphi_{c,p} : T \to 2^{k+p}\Z$, 
generically Nash constructible  in dimension $k$,  such that 
\begin{equation}\label{NFdef}
 c= \{ x \in  T\ ;\ \varphi_{c,p}  (x) \notin   2^{k+p+1}\Z \}, 
\end{equation}  
up to a set of dimension less than $k$. The choice of $\varphi _{c,p}$ is not unique.  
Let $Z$ denote the Zariski closure of $\supp c$.  By multiplying $\varphi _{c,p}$ by $\1_Z$, we may always assume 
that $\supp \varphi \subset Z$ and hence, in particular, that $\dim \supp \varphi _{c,p}\le k$.  

We say that $c\in C_k (T)$ is \emph{pure} if $c\in \Na_{-k} C_k (T)$. 
By \cite{aussois}, Theorem 3.9, and the existence of arc-symmetric closure, cf. \cite 
{kurdyka1}, \cite{aussois}, $c\in  C_k (T)$ is pure if and only if $\supp c$ coincides with an $\AS$ set
 (up to a set of dimension smaller than $k$).  For $T$ compact this means that $c$ is pure if and only if $c$ 
can be represented by an arc-symmetric set.  
By \cite{mccroryparusinski}, if $\dim T=k$ then  every semialgebraically constructible function
 $\varphi : T \to 2^{k}\Z$ is Nash constructible.  
 Hence $\Na_0 C_k (T) =  C_k( T)$.

The boundary operator preserves the Nash constructible filtration: 
$$
\partial \Na_p C_k(T)\subset \Na_p C_{k-1}(T).
$$
Indeed, if $c\in C_k (T)$ is given by \eqref{NFdef} and $\dim \supp \varphi _{c,p}\le k$., then 
\begin{equation} \label{NFpartial}
\partial c= \{ x \in  Z\ ;\ \varphi_{\partial c,p} (x) \notin   2^{k+p}\Z \} , 
\end{equation}  
where $\varphi_{\partial c,p}$ equals  $\frac 1 2 \Lambda \varphi_{c,p}$  for $k$ odd and 
  $\frac 1 2 \Omega \varphi_{c,p}$ for $k$ even, \emph{cf.} \cite{mccroryparusinski}.  
  A geometric interpretation of this formula is as follows ; \emph{cf.}\ \cite{bonnard}.  Let $Z$ be the Zariski closure of $\supp c$, so $ \dim Z =k$ if $c\ne 0$.  Let $W$ be an  algebraic subset of $Z$ such that $\dim W<k$ and $\varphi_{c,p}$ is 
locally constant on $Z\setminus W$.  At a generic point $x$ of $W$, we define 
$\partial_W \varphi_{c,p}(x) $ as the average of the values of $\varphi_{c,p}$ on the  local connected components of $Z\setminus W$ at $x$.  It can be shown that $\partial_W \varphi_{c,p}(x) $ is generically Nash constructible in dimension $k-1$. (For $k$ odd it equals $(\frac 1 2 \Lambda \varphi_{c,p})|_W$ and for $k$ even it equals $(\frac 1 2 \Omega \varphi_{c,p})|_W$; \emph{cf.} \cite{mccroryparusinski}.)  

We say that a square in  $\mathcal X_{\AS}$
 \begin{align}\label{ASacyclic}
 \minCDarrowwidth 1pt\begin{CD}
\tilde S @> >> \tilde T \\
 @V VV @VV\pi V \\
S @>i >> T
\end{CD}\end{align}
is acyclic if $i$ is a closed inclusion, $\tilde S =\pi\inv (Y)$ and the restriction of $\pi$ is a homeomorphism 
$\tilde T \setminus \tilde S \to T\setminus S$.   

\begin{thm} \label{Nashproperties} 
The  functor  $\Na C_* : \mathcal X_{\AS} \to \Cat$,  defined on 
the category $\mathcal X_{\AS}$ of locally compact $\AS$ sets and continuous proper maps 
with $\AS$ graphs, satisfies:
\begin{enumerate}
\item 
For an acyclic square \eqref{ASacyclic}  
the sequences 
\begin{eqnarray*}
& & 0\to \Na_p  C_{k} (\tilde S) \to \Na_p C_k (S) \oplus \Na_p  C_{k} (\tilde T) \to 
\Na_p C_{k} (T) \to 0  \\
& & 0\to \frac { \Na_p C_k(\tilde S) } {\Na_{p-1}  C_{k} (\tilde S)} \to 
\frac {\Na_p C_k(S)} {\Na_{p-1} C_k (S)} \oplus  \frac {\Na_p C_k(\tilde T)} {\Na _{p-1}  C_{k} (\tilde T)} \to 
\frac {\Na_p C_k(T) } {\Na_{p-1} C_{k} (T)} \to 0  
\end{eqnarray*}
are exact.
\item 
 For a closed inclusion $S\subset T$,  the restriction to $U=T\setminus S$ induces a  morphism of filtered complexes  
 $\Na C_*(T) \to \Na C_* (U)$, and the sequences 
\begin{eqnarray*} \label {Nexact1*}
& & 0\to \Na_p  C_{k} (S) \to \Na_p  C_{k} (T) \to 
\Na_p C_{k} (U) \to 0 \\ 
\label {Nexact2*}
& & 0\to \frac {\Na_p C_k(S)} {\Na _{p-1}  C_{k} (S)} \to \frac {\Na _p C_k(T)} {\Na_{p-1}  C_{k} (T)} \to 
\frac {\Na_p C_k(U)} {\Na_{p-1}  C_{k} (U)} \to 0  
\end{eqnarray*}
are exact.  
\end{enumerate}   
\end{thm}

\begin{proof}
We first show that $\Na C_*$ is a functor; that is, for a proper morphism $f:T\to S$, 
$f_* \Na_p C_k (T) \subset \Na_p C_k(S)$. 
Let $c\in \Na_p C_k (T)$ and let $\varphi = \varphi_{c,p} $ be a Nash constructible function on $T$ 
satisfying \eqref{NFdef}     
(up to a set of dimension $<k$). Then 
$$
f_* c = \{ y\in S\ ;\ f_* (\psi) (y) \notin 2^{k+p+1} \Z \}; 
$$
 that is,  $ \varphi_{f_*c,p} = f_*  \varphi_{c,p} $.

For a closed inclusion  $S\subset T$, the restriction to $U =T\setminus S$ of a Nash constructible function on 
$T$ is Nash 
constructible.  Therefore the restriction defines a morphism  $\Na C_*(T) \to \Na C_* (U)$.  
The exactness of the first sequence of (2) can be verified easily by direct computation.   We note, moreover, that for fixed $k$ the morphism 
$$\Na_* C_k (T) \to \Na_* C_k (U)$$
splits (the splitting does not commute with the boundary), by assigning to $c\in \Na_p C_k(U)$ its closure 
$\bar c\in C_k(T)$.   Let $\varphi :T \to 2^{k+p}\Z$ be a Nash constructible function 
such that $\varphi|_{T\setminus S}= \varphi_{c,p}$.  
Then $\bar c= \{ x \in  T\ ;\ (\1_T-\1 _S) \varphi (x) \notin   2^{k+p+1}\Z \} $ up to a set of dimension $<k$.  

The exactness of the second sequence of (2) and the sequences of (1) now follow by standard arguments.  
(See the proof of Theorem \ref{geomproperties}.)
\end{proof}


\subsection{The Nash constructible filtration for Nash manifolds} \label{Nashmanifolds}
A \emph{Nash function} on an open semialgebraic subset $U$ of $\R^N$ is a real analytic semialgebraic function.  Nash morphisms and Nash manifolds play an important role in real algebraic geometry.  
In particular a connected component of  compact  nonsingular real algebraic subset of $\R^n$ is a Nash submanifold of $\R^N$
 in the sense of \cite{BCR} (2.9.9).  Since 
$\R \proj ^n$ can be embedded in $\R^N$ by a rational diffeomorphism (\cite{akbulutking}, 
\cite{BCR} (3.4.2)) the connected components of nonsingular projective real algebraic varieties can 
be considered as Nash submanifolds of affine space. 
 By the Nash Theorem (\emph{cf.}\ \cite {BCR} 14.1.8),   every compact $C^\infty$ manifold is $C^\infty$-diffeomorphic to a 
 Nash submanifold of an affine space, and moreover such a model is unique up to 
 Nash diffeomorphism (\cite{BCR} Corollary 8.9.7).   In what follows by a \emph{Nash manifold} we mean a compact Nash submanifold of an affine space.

Compact  Nash manifolds and the  
 graphs of Nash morphisms on them are  $\AS$ sets. 
 If $N$ is a Nash manifold, the Nash constructible filtration is contained in the canonical filtration,
\begin{equation}\label{arcfiltr}
\Na_p C_k (N) \subset F^{can}_p C_k (N),
\end{equation} 
since $\partial_k(\Na_{-k}C_k(N))=0$.
Thus on the category of Nash manifolds and Nash maps have a morphism of functors
$$
\tau: \Na C_*  \to F^{can} C_*.
$$

\begin{thm}\label{isoNash}
For every Nash manifold  $N$,   
$$
\tau(N): \Na C_*(N)  \to F^{can} C_*(N)
$$ 
is a filtered quasi-isomorphism.  
\end{thm}

\begin{proof} 
We show that for all $p$ and $k$, $\tau(N)$ induces
 an isomorphism 
\begin{equation}
\tau_*: H_k ( \Na_p C_* (N))  \cong H_k (F^{can}_p C_*(N)).
\end{equation}
Then, by the long exact homology sequences of $(\Na_p C_* (N),\Na_{p-1} C_* (N))$ and 
$ (F^{can}_p C_*(N), $ $ F^{can}_{p-1} C_*(N))$, 

$$
\tau_*: H_k \left(\frac { \Na_p C_* (N)} {\Na_{p-1} C_* (N) }\right) \to H_k \left(\frac {F^{can}_p C_*(N)} { F^{can}_{p-1} C_*(N)}\right )
$$
is an isomorphism, which shows the claim of the theorem.  

We proceed by  induction on the dimension of $N$.  
  We call 
a Nash morphism $\pi :\tilde N \to N$ a \emph{Nash multi-blowup} if $\pi$ is a composition 
of blowups along nowhere dense Nash submanifolds.  

\begin{prop}\label{NashMikhalkin}
Let $N, N'$ be compact connected Nash manifolds of the same dimension.  Then there exist  multi-blowups 
$\pi :\tilde N \to N$, $\sigma :\tilde N' \to N'$ such that $\tilde N$ and $\tilde N'$ are Nash diffeomorphic.  
\end{prop}

\begin{proof}

By a theorem of Mikhalkin \cite {mikhalkin1},   \cite {mikhalkin2} Proposition 2.6, 
 any two connected closed $C^\infty$ 
manifolds of the same dimension can be connected by a sequence of 
$C^\infty$ blowups and and then blowdowns with smooth centers.  We show that this $C^\infty$ statement implies an analgous statement in the Nash category.  

Let $M$ be a closed $C^\infty$ manifold.  By the Nash-Tognoli Theorem there is a nonsingular real algebraic set 
 $X$, \emph{a fortiori} a Nash manifold, that is $C^\infty$-diffeomorphic 
to $M$.  Moreover, by approximation by Nash mappings, any two Nash models of $M$ are Nash diffeomorphic; see \cite{BCR} Corollary 8.9.7.   Thus in order to show Proposition 
\ref{NashMikhalkin} we need only the following lemma.

\begin{lem}
Let  $C\subset M$ be a $C^\infty$ submanifold of a closed $C^\infty$ manifold $M$.  Suppose that $M$ is $C^\infty$-diffeomorphic to a Nash 
manifold $N$.  Then there exists a Nash submanifold $D\subset N$ such that the blowups
$Bl(M,C)$ of $M$ along $C$   and 
$Bl(N,D)$ of $N$ along $D$ are $C^\infty$-diffeomorphic.  
\end{lem}

\noindent \emph{Proof.}
By the relative version of 
Nash-Tognoli Theorem proved by Akbulut-King and Benedetti-Tognoli (see for instance \cite{BCR} Remark 14.1.15),  
there is a nonsingular real algebraic set $X$ and a $C^\infty$ diffeomorphism $\varphi : M\to X$ such 
that $Y=\varphi (C)$ is a nonsingular algebraic set.   Then the blowups  $Bl(M,C)$ of $M$ along $C$   and 
$Bl(X,Y)$ of $X$ along $Y$ are $C^\infty$-diffeomorphic.  Moreover, since $X$ and $N$ are 
$C^\infty$-diffeomorphic, they are Nash diffeomorphic by a Nash diffeomorphism $\psi : X \to N$.  
Then $Bl(X,Y)$ and $Bl(N, \psi(Y))$ are Nash diffeomorphic. 
This proves the Lemma and the Proposition.   
\end{proof}

\begin{lem}\label{Nashblowup}
Let $N$ be a compact connected Nash manifold  and let $\pi : \tilde N \to N$ denote 
the blowup of $N$ along a nowhere dense Nash submanifold $Y$.  Then  $\tau(N)$ is a quasi-isomorphism 
 if and only if $ \tau(\tilde N)$ is a quasi-isomorphism.  
\end{lem} 

\begin{proof}
Let $\tilde Y= \pi \inv (Y)$ denote the exceptional divisor of $\pi$.  For each $p$ consider the diagram 
\begin{equation*}\minCDarrowwidth 1pt\begin{CD}
  @>>> H_{k+1} ( \Na_p C_* (N))  @>>> H_{k} ( \Na_p C_* (\tilde Y))  @>>> H_{k} ( \Na_p C_* (Y))  
\oplus H_{k} ( \Na_p C_* (\tilde N)) @>>>  
 \\ @. @VVV 
 @VVV @VVV  @.\\
  @>>> H_{k+1} ( F^{can}_p C_* (N))  @>>> H_{k} ( F^{can}_p C_* (\tilde Y))  @>>> H_{k} ( F^{can}_p C_* (Y))  
\oplus H_{k} ( F^{can} C_* (\tilde N)) @>>>  
\end{CD}\end{equation*}  
The top row is exact by Theorem \ref{Nashproperties}. For all manifolds $N$ and for all $p$ and $k$, we have
$$
H_{k} ( F^{can}_p C_* (N)) =  \begin{cases} 
  H_k(N) \qquad   k\geq -p \\
  \, 0  \qquad \quad  \quad  k<-p,
  \end{cases}
$$
so the short exact sequences \eqref{short} give that the bottom row is exact.
The lemma now follows from the inductive assumption and the Five Lemma.  
\end{proof}

Consequently it suffices to show that $\tau(N)$ is a quasi-isomorphism  for a single 
connected Nash manifold of each dimension $n$.  
We check this assertion for the standard sphere $S^n$ by showing that 
$$
H_{k} ( \Na_p C_* (S^n)) =  \begin{cases} 
  H_k (S^n) \quad  \text { if $k = 0$ or $n$ and } p\geq - k \\
  0  \quad \qquad \quad  \text { otherwise}. 
  \end{cases}
$$

Let $c\in \Na_pC_k (S^n)$, $k<n$, be a cycle described as in \eqref{NFdef} by  the Nash constructible function  
$\varphi_{c,p} : Z \to 2^{k+p} \Z$, where  $Z$ is the Zariski closure of $\supp c$.  Then $c$ can 
be contracted to a point.  More precisely, choose $p\in S^n\setminus Z$.  
Then $S^n \setminus \{p\}$ and $\R^n$ are isomorphic.   Define a Nash constructible function  $\Phi : Z\times \R \to 2^{k+p+1} \Z$ by the formula 
  \begin{equation*} 
  \Phi (x,t) = \begin{cases} 
  2 \varphi_{c,p} (x) \quad  \text { if } t\in [0,1] \\
  0  \qquad  \qquad \text { otherwise.} 
  \end{cases}
  \end{equation*}
  Then
  $$
  c\times [0,1] = \{(x,t) \in Z\times \R\ ;\ \Phi (x,t)\notin 2^{k+p+2} \Z \};
  $$
  so $c\times [0,1] \in \Na_pC_{k+1} (Z\times \R)$.  
  The morphism $f: Z \times \R \to \R^n$, $f(x,t) = tx$, is proper and for $k>0$ 
  $$
  \partial f_*( c\times [0,1]) =  f_*(  \partial c\times [0,1]) = c, 
  $$
  which shows that $c$ is a boundary in $\Na _p C_* (S^n)$.  
  If  $k=0$ then $\partial f_*( c\times [0,1]) =   c  - (\deg c) [0]$.  
  
If $c\in \Na_p C_n (S^n)$ is a cycle, then $c$ is a cycle in $ C_n (S^n)$; that is, either $c=0$ or $c=[S^n]$.
This completes the proof of Theorem \ref{isoNash}. 
\end{proof}


\subsection{Consequences for the weight filtration} \label{Nashweight}

\begin{cor}\label{iso}
For every real algebraic variety $X$ 
the localization of $\tau$ induces a quasi-isomorphism 
$\tau' (X): \Na C_* (X) \to \mathcal WC_*(X)$.   
\end{cor}

\begin{proof}
Theorem \ref{Nashproperties} yields that 
the functor $\Na C_*:\Sch \to \HCat$
satisfies properties (1) and (2) of Theorem \ref{navarro}.  Hence Theorem \ref{isoNash}
and Theorem \ref{naturality} give the desired result. 
\end{proof}

\begin{cor}\label{2filtrations}
Let $X$ be a real algebraic variety.  Then for all $p$ and $k$, $\Na _p C_k(X) = \A _p C_k(X) $.
\end{cor}

\begin{proof}
We show that the Nash constructible filtration satisfies properties (1) and (2) of  Theorem 
\ref{axioms}.  This is obvious for property (1).  We show property (2). Let 
$\tilde c =\pi \inv (c)$.  First we note that 
$$
c \in \Na_p C_k(X) \Leftrightarrow \tilde c \in \Na_p C_k(\tilde X) .
$$
Indeed, ($\Leftarrow$) follows from functoriality, since $c=\pi_*(\tilde c)$.  If $c$ is given by 
\eqref{Nashfiltration} then $\pi^* (\varphi_{c,p})$ is Nash constructible and describes $\tilde c$.
Thus it suffices to show 
$$
\tilde c \in \Na_p C_k(\tilde X)  \Leftrightarrow \partial \tilde c \in \Na_p C_{k-1}(\tilde X) 
$$
for $p\ge -k$, with the implication ($\Rightarrow$) being obvious.  
If $p=-k$ then each cycle is arc-symmetric. (Such a cycle is a union of connected components of $\tilde X$, since 
$\tilde X$ is nonsingular and compact.)  
For  $p>-k$ suppose, contrary to our claim,  that 
\begin{equation*}\label{contra}
\tilde c \in  \Na_p C_k(\tilde X) \setminus  \Na_{p-1} C_k (\tilde X)  \quad \text { and } \quad 
\partial \tilde c \in \Na_{p-1} C_{k-1}(\tilde X) .  
\end{equation*}
By Corollary \ref{iso} and Proposition \ref{purity}   
$$
H_k\left( \frac {\Na_p C_*(\tilde X)} {\Na_{p-1} C_*(\tilde X)} \right)=0,
$$
and $\tilde c$ has to be a relative boundary.  But $\dim \tilde X=k$ and $C_{k+1} (\tilde X)=0$.  
This completes the proof.  
\end{proof}



\section{Applications to real algebraic and analytic geometry}  

Algebraic subsets of affine space, or more generally 
 $Z$-open or $Z$-closed 
affine or projective sets in the sense of Akbulut and King \cite{akbulutking}, are  $\AS$ sets.  
So are the graphs of regular rational mappings.   Therefore Theorems \ref{Nashproperties}  and  \ref{isoNash} give the following result.

\begin{thm}\label{raag}
The Nash constructible filtration of closed semialgebraic chains defines a functor from the category of affine real algebraic sets 
and proper regular rational mappings to the category of bounded chain complexes of $\Z_2$ vector spaces with increasing bounded filtration.

This functor is additive and acyclic; that is, it satisfies 
properties (1) and (2) of Theorem \ref{Nashproperties}; and it
induces the weight spectral sequence and the weight filtration on Borel-Moore homology with coefficients in $\Z_2$.   

For compact nonsingular algebraic sets,  the reindexed  weight spectral sequence is pure: $\tilde E^2_{p,q} = 0$ for $p>0$.
\end{thm}

For the last  claim of the theorem we note that every compact affine real 
algebraic set that is  nonsingular in the sense of  \cite{akbulutking} and
\cite{BCR} admits a compact nonsingular complexification.  Thus the claim follows from Theorem \ref{isoNash}.   

 The purity of $\tilde E^2$ implies the purity of 
$\tilde E^\infty$:  $\tilde E^\infty_{ p,  q} =0$ for $ p >0$. 
Consequently every nontrivial homology class  
 of a nonsingular compact affine or projective real algebraic  variety 
 can be represented by a semialgebraic arc-symmetric set,  a result proved
 directly in  \cite{kucharz} and \cite{aussois}.

\rem Theorem \ref{Nashproperties} and Theorem \ref{isoNash} can be used in more general contexts. 
A compact real analytic semialgebraic subset of a real algebraic variety is an $\AS$ set.  
A compact semialgebraic set that is the graph of a real analytic map, or more generally the graph  of an arc-analytic 
mapping (\emph{cf.}\ \cite{aussois}), is arc-symmetric.  
In section \ref{Nashweight} we have already used  that compact affine  Nash manifolds and 
graphs of Nash morphisms defined on compact Nash manifolds are arc-symmetric.    

\vskip.1in
   
The weight filtration of homology  
is an isomorphism invariant but not a homeomorphism invariant; this is  discussed in \cite{virtual} for 
the dual weight filtration of cohomology.  
   
 \begin{prop}\label{homeo}
 Let $X$ and $Y$ be locally compact $\AS$ sets, and let $f:X\to Y$ be a homeomorphism  
 with $\AS$ graph.  Then $f_*:\Na C_*(X)\to\Na C_*( Y)$  is  an isomorphism 
of filtered complexes.   

Consequently, $f_*$ induces an isomorphism of the weight spectral sequences of $X$ and $Y$ and of the weight filtrations of $H_*(X)$ and $H_*(Y)$.  Thus the virtual Betti  numbers 
\eqref{virtualbetti} of $X$ and $Y$ are equal. 
\end{prop} 

\begin{proof}
The first claim follows from the fact that $\Na C_* : \X_{\mathcal A\mathcal S} \to \Cat$ is a functor; see the proof of  Theorem \ref{Nashproperties}.   The rest of the proposition then follows from 
  Theorem \ref{Nashproperties}    and Theorem \ref{isoNash}. 
\end{proof}

\rem Propostion \ref{homeo} applies, for instance, to regular homeomorphisms such as $f:\R \to \R$, $f(x) = x^3$.  
The construction of the virtual Betti numbers of \cite{virtual} was extended to $\AS$ sets by G. Fichou in 
\cite{fichou}, where their invariance by Nash diffeomorphism was shown.  The arguments of  \cite{virtual} 
and \cite{fichou} use the weak factorization theorem of \cite{AKMW}.


\subsection{The virtual Poincar\'e polynomial}  Let $X$ be a locally compact $\AS$ set.  The virtual Betti 
numbers give rise to the \emph{virtual Poincar\'e polynomial }
\begin{equation}
\beta (X) = \sum _i  \beta_i(X) \, t^i.  
\end{equation}
For real algebraic varieties the virtual Poincar\'e polynomial was first introduced in \cite{virtual}.  
For $\AS$ sets,  not necessarily locally compact, it was defined in  \cite{fichou}.  It satisfies the following properties (see \cite{virtual}, \cite{fichou}):
\begin{enumerate}
\item
\emph{Additivity:} For finite disjoint union $X= \sqcup X_i$, $\beta(X) = \sum \beta (X_i)$. 
\item
\emph{Multiplicativity:}  $\beta (X\times Y) =  \beta(X) \cdot \beta (Y)$.  
\item
\emph{Degree:}  For $X\ne \emptyset$, $\deg \beta (X) =  \dim X$ and the leading coefficient $\beta(X)$ is  strictly 
positive.  
\end{enumerate}
(If $X$ is not locally compact we can decompose it into a finite disjoint union of locally compact $\AS$ sets 
$X= \sqcup X_i$ and 
define $\beta (X) = \sum \beta (X_i)$.)

We say that a function $X\to e(X)$ defined on real algebraic sets is an \emph{invariant}
 if it an isomorphism 
invariant, that is $e(X)=e(Y)$ if $X$ and $Y$ are isomorphic (by a biregular rational mapping).   We say that $e$ is additive, resp. multiplicative, if $e$ takes values in an abelian group and $e(X\setminus Y) = e(X) -e(Y)$ for all $Y\subset X$, resp. $e$ takes values in a ring and 
$e(X\times Y) = e(X) e(Y)$ for all $X,Y$.  
The following theorem states that the virtual Betti polynomial is a universal additive, or additive and multiplicative, invariant defined 
on real algebraic sets (or real points of real algebraic varieties in general), among those invariants that 
do not distinguish Nash diffeomorphic compact nonsingular 
real algebraic sets.  

\begin{thm}
Let $e$ be an additive invariant defined on real algebraic sets.  
Suppose that for every pair $X, Y$ of 
Nash diffeomorphic nonsingular 
compact real algebraic sets we have $e(X)=e(Y)$.  
Then there exists a unique group homomorphism $h_e : \Z[t] \to G$ such that $e = h_e \circ \beta$. If, moreover, $e$  is multiplicative  then $h_e$ is a ring homomorphism.   
\end{thm}

\begin{proof}
Define $h(t^n) = e(\R^n)$.  We claim that the additive invariant $\varphi(X) = h(\beta (X)) - e(X)$ vanishes for every real algebraic set $X$.   This is the case for $X=\R^n$ since $\beta (\R^n)=t^n$.  
By additivity, this is also the case for $S^n=\R^n \sqcup pt$.  
By the existence of an algebraic compactification and resolution of singularities, it suffices to show the claim for compact nonsingular real 
algebraic sets.

Let $X$ be a compact nonsingular real 
algebraic set and let $\tilde X$ be the blowup of $X$ along a smooth nowhere dense center.  Then, 
using induction on $\dim X$, we see that $\varphi(X) =0$ if and only if $\varphi (\tilde X) =0$.  
By the relative version of the
Nash-Tognoli Theorem,  the same result holds if we have that $\tilde X$ is Nash diffeomorphic to the blowup of 
a nowhere dense Nash submanifold of $X$.  
Thus the claim and hence the first statement follows from Mikhalkin's Theorem.  
\end{proof}

Following earlier results of Ax and Borel, K. Kurdyka showed in  \cite {kurdyka2} that any regular injective self-morphism 
$f:X\to X$ of a real algebraic variety is surjective.  It was then showed in \cite{parusinski} that an injective
 continuous self-map 
$f:X\to X$  of a locally compact $\AS$ set, such that the graph of $f$ is an $\AS$ set, is a homeomorphism.  The arguments of both \cite{kurdyka2} and \cite{parusinski} are topological and use the continuity of $f$ in essential way.  
The use of additive invariants allows us to handle the non-continuous case.

\begin{thm}
Let $X$ be an $\AS$ set and let $f:X\to X$ be a map with $\AS$ graph.  If $f$ is injective then it is surjective. 
\end{thm}

\begin{proof}
It suffices to show that there exists a finite decomposition $X= \sqcup X_i$ into locally compact $\AS$ sets such that 
for each $i$, $f$ restricted to $X_i$ is a homeomorphism onto its image.  Then, by Corollary \ref{homeo},  
$$
\beta(X\setminus \sqcup_i f(X_i)) = \beta(X) - \sum_i\beta(X_i) =0,
$$
and hence, by the degree property, $ X\setminus \sqcup_i f(X_i)=\emptyset$.  

To get the required decomposition first we note that by classical theory there exists a semialgebraic stratification of $X=\sqcup S_j$ such that $f$ restricted to each stratum is real analytic.  We show that we may 
choose strata belonging to the class $\AS$.  (We do not require the strata to be connected.) 
By \cite{kurdyka2}, \cite{parusinski},  each semialgebraic subset $A$ of a real algebraic 
variety $V$ has a minimal $\AS$ closure in $V$, denoted $\overline A ^{\AS}$.  Moreover if $A$ is $\AS$ then  
$\dim  \overline A ^{\AS} \setminus A < \dim A$.   Therefore, we may take as the first subset of 
the decomposition the complement in $X$ of the $\AS$ closure of the union of strata $S_j$ of dimension $<\dim X$, 
and then proceed by induction on dimension.  

Let  $X=\sqcup S_j$ be a stratification with $\AS$ strata and such that $f$ is analytic on each stratum.  
Then, for each stratum $S_j$,  we apply the above argument to $f\inv$ defined on $f(S_j)$.   The induced 
subdivision of $f(S_j)$, and hence of $S_j$, satisfies the required property.  
\end{proof}

Of course, in general, surjectivity does not apply injectivity for a self-map.  Nevertheless we have the following 
result.  

\begin{thm}
Let $X$ be an $\AS$ set and let $f:X\to X$ be a surjective map with $\AS$ graph.  Suppose that 
 there exist a finite $\AS$ decomposition $X= \sqcup Y_i$ and $\AS$ sets $F_i$ such that for each 
$i$, $f\inv (Y_i)$ is homeomophic to $Y_i\times F_i$ by a homeomorphism with $\AS$ graph.  
Then $f$ is injective. 
\end{thm}

\begin{proof}
We have 
$$
0 =  \beta (X) -  \beta (f(X))  = \sum \beta(Y_i) (\beta(F_i)-1) . 
$$
Then for each $i$,  $\beta(F_i)-1 =0$, otherwise the polynomial on the right-hand side would be nonzero with strictly positive leading coefficient.    
\end{proof}


\subsection{Application to spaces of orderings}  
Let $V$ be an irreducible real algebraic subset of $\R^N$.  A function $\varphi : V \to \Z$ is 
called \emph{algebraically constructible} if it satisfies one of the following equivalent properties 
(\emph{cf.}\ \cite{mccroryparusinski}, \cite{parusinskiszafraniec}):
\begin{enumerate}
\item
There exist a finite family of proper regular morphisms $f_i:Z_i\to V$, 
and integers $m_i$,  such that for all $x\in V$, 
\begin{equation} \label{algconstr}
\varphi (x) = \sum_i m_i \chi ( 
f_{i}\inv (x)\cap Z_i).
\end{equation}
\item
There are finitely many polynomials $P_i\in \R[x_1\dots, x_N]$ such that for all $x\in V$,
$$
\varphi (x) = \sum _i  \sgn P_i(x).
$$
\end{enumerate}
Let $K=K(V)$ denote the field 
of rational functions of $V$.  A function $\varphi : V \to \Z$ is   generically 
algebraically constructible if and only if can be identified, up to a set of dimension smaller $\dim V$, with 
the signature of a quadratic form over $K$.  Denote by $\X$ the real spectrum of $K$.  A (semialgebraically) constructible function on $V$, up to a set of dimension smaller $\dim V$, 
can be identified with a continuous function $\varphi : \X \to \Z$ (\cite {BCR} ch.\ 7, \cite {marshall}, \cite{bonnard2}).  
The representation theorem 
of Becker and Br\"ocker gives a fan criterion for recognizing  generically algebraically constructible function on $V$.  The following two theorems are due to I. Bonnard.

\begin{thm}\label{fan1}\label{representation} {\rm (\cite{bonnard2}) }
A constructible function $\varphi : V \to \Z$ is generically algebraically constructible  if 
and only for any finite fan $F$ of $\X$
\begin{equation}\label{fancriterion}
\sum _{\sigma \in F} \varphi (\sigma ) \equiv 0 \mod |F|.
\end{equation}
\end{thm}

For the notion of a fan see  \cite {BCR} ch.\ 7, \cite {marshall}, \cite{bonnard2}.  The number of elements $|F|$ 
of a finite fan $F$ is always a power of $2$.  
It is known that for every finite fan $F$ of $\X$ there 
exists a valuation ring $B_F$ of $K$ compatible with $F$, and on whose residue field the fan $F$ induces exactly one or two distinct orderings.  
Denote by  $\mathcal F$ the set  of these fans of $K$ for which the  residue field induces only one ordering.  

\begin{thm} \label{fan2}  {\rm (\cite{bonnard}) }
A constructible function $\varphi : V \to \Z$ is generically Nash constructible if and only if  \eqref{fancriterion} holds 
for every fan $F\in \mathcal F$.
\end{thm}
 
The following question is due to M. Coste and M. A. Marshall (\cite{marshall} Question 2):

\vskip.1in \emph{Suppose that a constructible function $\varphi : V\to \Z$ satisfies \eqref{fancriterion} for every fan $F$ of $K$ 
with $|F| \le 2^n$.  Does there exists a generically algebraically constructible function $\psi : V\to \Z$ such that 
for each $x\in V$, $ \varphi (x)  - \psi (x) \equiv 0 \mod 2^n $?}

\vskip.1in We give a positive answer to the Nash constructible analog of this question.

\begin{thm}\label{question2}
Suppose that a constructible function $\varphi : V\to \Z$ satisfies \eqref{fancriterion} for every fan $F \in \mathcal F$ 
with $|F| \le 2^n$.  Then there exists a generically Nash constructible function $\psi : V\to \Z$ such that 
for each $x\in V$, 
$ \varphi (x)  - \psi (x) \equiv 0 \mod 2^n$.
\end{thm}

\begin{proof}
We proceed by induction on $n$ and on $k =\dim V$.  The case $n=0$ is trivial.  

Suppose $\varphi : V\to \Z$ satisfies \eqref{fancriterion} for every fan $F \in \mathcal F$ 
with $|F| \le 2^n$, $n\ge 1$.  By the inductive assumption, $\varphi$ is congruent modulo $2^{n-1}$ to a generically 
Nash constructible function $\psi_{n-1}$.  By replacing $\varphi$ by $\varphi - \psi_{n-1}$, we may suppose $2^{n-1}$ divides 
$\varphi$.  

We may also suppose $V$ compact and nonsingular, just choosing a model for $K=K(V)$.  Moreover, by resolution of singularities, we may assume that $\varphi $ is constant in the complement of a normal crossing 
divisor $D=\bigcup D_i\subset V$.  

Let $c$ be given by 
\eqref{NFdef} with $\varphi_{c,p}=\varphi$ and $p= n-k-1$.    
At a generic point $x$ of $D_i$  define 
$\partial_{D_i} \varphi(x) $ as the average of the values of $\varphi$ on the  local connected components of $V\setminus D$ at $x$.  Then $\partial c = \sum_i \partial _{i} c$, where $ \partial _{i} c$ is described 
by $\partial_{D_i} \varphi$ as in \eqref{NFpartial} (see \cite{bonnard}).  
Note that the constructible functions  $\partial_{D_i} \varphi$  satisfy
the inductive assumption for $n-1$.  Hence  each $\partial_{D_i} \varphi$ is 
congruent to a generically Nash constructible function modulo $2^{n-1}$.  In other words 
 $\partial c\in \Na_p C_{k-1}(V)$.  Then by Corollary 
\ref{2filtrations}  we have  $c\in \Na_p C_k(V)$, which implies the statement of the theorem.  
\end{proof}

Using Corollary \ref{2filtrations} we obtain the following result. The original proof was based on the fan criterion (Theorem \ref{fan2}).  

\begin{prop} \label{Nashdivisorcriterion} {\rm (\cite{bonnard}) }
Let $V\subset \R^N$ be compact, irreducible, and nonsingular.  Suppose that the constructible function  
$\varphi : V\to \Z$ is constant in the complement of a normal crossing 
divisor $D=\bigcup D_i\subset V$.   Then $\varphi$ is generically Nash constructible if and only if $\partial_D \varphi$ is generically Nash constructible.  
\end{prop}

\begin{proof}
We show only  ($\Leftarrow$).  Suppose $2^{k+p} | \varphi$ generically, where $k =\dim V$, and let $c$ be given by 
\eqref{NFdef} with $\varphi_{c,p}=\varphi$.  Then by our assumption  $\partial c\in \Na_p C_{k-1}(V)$.  
By Corollary 
\ref{2filtrations}  we have  $c\in \Na_p C_k(V)$, which shows that, modulo $2^{k+p+1}$, $\varphi$ coincides with a 
generically Nash constructible function $\psi$.  Then we apply the same argument to $\varphi -\psi$.  
\end{proof}

\rem We note that the above Proposition implies neither Theorem \ref{question2} nor Corollary \ref{2filtrations}.    
Similarly the analog of Proposition \ref{Nashdivisorcriterion}, proved in \cite{bonnard2}, does not give 
an answer to Coste and Marshall's question.





\section{The toric filtration}
\label{toric}

In their 
investigation of the relation between the homology of the real and complex points of a toric variety \cite{BFMH},
Bihan \emph{et al.}\ define a filtration on the  cellular chain complex of a real toric variety. We  prove
that this filtered complex is quasi-isomorphic to the semialgebraic chain complex with the Nash constructible
filtration. Thus the toric filtered chain complex realizes the weight complex, and the real toric spectral sequence of \cite{BFMH}
is isomorphic to the weight spectral sequence.

For background on toric varieties see \cite{fulton}. We use a simplified version of the notation of \cite{BFMH}. 
Let  $\Delta$ be a rational fan in $\R^n$, and let $X_\Delta$ be the real
toric variety defined by $\Delta$. The group $\T=(\R^*)^n$ acts on $X_\Delta$, and the $k$-dimensional
orbits $\mathcal O_\sigma$ of this action correspond to the codimension $k$ cones $\sigma$ of $\Delta$. 

The positive part 
$X^+_\Delta$ of $X_\Delta$ is a closed semialgebraic subset of $X_\Delta$, and there is a canonical retraction
$r:X_\Delta\to X^+_\Delta$ that can be identified with the orbit map of the action of the finite group $T = (S^0)^n$ on $X_\Delta$,
where $S^0=\{-1,+1\}\subset \R^*$. The $T$-quotient of the $k$-dimensional $\T$-orbit $\mathcal O_\sigma$ 
is a semialgebraic $k$-cell $c_\sigma$ of $X^+_\Delta$, and $\mathcal O_\sigma$ is a disjoint union of  $k$-cells,
each of which maps homeomorphically onto $c_\sigma$ by the quotient map. This decomposition defines a cell structure on $X_\Delta$  such that $X_\Delta^+$
is a subcomplex and the 
quotient map is cellular. Let $C_*(\Delta)$ be the cellular chain complex of $X_\Delta$ with coefficients in $\Z_2$. The closures of the cells 
of $X_\Delta$ are not necessarily compact, but they are  semialgebraic subsets of $X_\Delta$. Thus we have a chain map 
\begin{equation}\label{chainmap}
\alpha:C_*(\Delta)\to C_*(X_\Delta)
\end{equation}
from cellular chains to semialgebraic chains.

The \emph{toric filtration} of the cellular chain complex $C_*(\Delta)$ is defined as follows \cite{BFMH}. For each $k\geq 0$
we define vector subspaces
\begin{equation}
\label{toricfiltration}
0 = \mathcal T_{-k-1 } C_k (\Delta) \subset \mathcal T_{-k} C_k (\Delta) \subset \mathcal T_{-k+1}  
C_k (\Delta) \subset \cdots  \subset \mathcal T_{0} C_k (\Delta) = C_k (\Delta),
\end{equation} 
such that $\partial_k(\mathcal T_pC_k(\Delta))\subset \mathcal T_pC_{k-1}(\Delta)$ for all $k$ and $p$.

Let $\sigma$ be a cone of the fan $\Delta$, with $\codim \sigma = k$.
Let $C_k(\sigma)$ be the  subspace of $C_k(\Delta)$ spanned by
the $k$-cells of $\mathcal O_\sigma$. Then
$$
C_k(\Delta) = \bigoplus_{\codim \sigma\ =\ k} C_k(\sigma).
$$
The orbit $\mathcal O_\sigma$ has a distinguished point $x_\sigma\in c_\sigma\subset X^+_\Delta$.
Let $T_\sigma = T/T^{x_\sigma}$, where $T^{x_\sigma}$ is the $T$-stabilizer of $x_\sigma$. We
identify the orbit $T\cdot x_\sigma$ with the multiplicative group $T_\sigma$. Each $k$-cell of $\mathcal O_\sigma$
contains a unique point of the orbit $T\cdot x_\sigma$. Thus we can make the identification
$C_k(\sigma) = C_0(T_\sigma)$, the set of formal sums $\sum_i a_i[g_i]$, where $a_i\in\Z_2$ and
$g_i\in T_\sigma$. The multiplication of $T_\sigma$ defines a multiplication on $C_0(T_\sigma)$, so
that $C_0(T_\sigma)$ is just the group algebra of $T_\sigma$ over $\Z_2$ .

Let $\mathcal I_\sigma$ be the augmentation ideal of the algebra $C_0(T_\sigma)$,
\begin{eqnarray*}
& & \mathcal I_\sigma = \Ker [\epsilon: C_0(T_\sigma)\to \Z_2], \\
& & \epsilon\sum_i a_i[g_i] = \sum_i a_i.
\end{eqnarray*}
For $p\leq 0$ we define $\mathcal T_pC_k(\sigma)$ to be the subspace
corresponding to the ideal $(\mathcal I_\sigma)^{-p} \subset C_0(T_\sigma)$, and we let
$$
\mathcal T_pC_k(\Delta) = \sum_{\codim \sigma\ =\ k}\mathcal T_pC_k(\sigma).
$$

If $\sigma < \tau$ in $\Delta$ and $\codim\tau = \codim\sigma - 1$, the geometry
of $\Delta$ determines a group homomorphism $\varphi_{\tau\sigma}:T_\sigma\to T_\tau$
(see \cite{BFMH}). Let $\partial_{\tau\sigma}: C_k(\sigma)\to C_{k-1}(\tau)$ be the induced algebra homomorphism.
We have $\partial_{\tau\sigma}(\mathcal I_\sigma)\subset \mathcal I_\tau$. 
The boundary map $\partial_k:C_k(\Delta)\to C_{k-1}(\Delta) $ is given by
$\partial_k(\sigma) = \sum_\tau\partial_{\tau\sigma}(\tau)$, and
$\partial_k(\mathcal T_pC_k(\Delta))\subset \mathcal T_pC_{k-1}(\Delta)$,
so $\mathcal T_pC_*(\Delta)$ is a subcomplex of $C_*(\Delta)$.

\begin{prop}\label{filtrations}
For all $k\geq 0$ and $p\leq 0$, the chain map $\alpha$ (\ref{chainmap}) takes the toric filtration (\ref {toricfiltration}) to the Nash filtration (\ref{Nashfiltration}),
$$
\alpha(\mathcal T_pC_k(\Delta))\subset\mathcal N_pC_k(X_\Delta).
$$
\end{prop}
\begin{proof}
It suffices to show that for every cone $\sigma\in \Delta$ with $\codim\sigma = k$,
$$
\alpha(\mathcal T_pC_k(\sigma))\subset \mathcal N_pC_k(\mathcal O_\sigma).
$$
The variety $\mathcal O_\sigma$ is isomorphic to $(\R^*)^k$, the toric variety of the trivial fan $\{0\}$ in $\R^k$,
and the action of $T_\sigma$ on $\mathcal O_\sigma$ corresponds to the action of $T_k = \{-1,+1\}^k$ on $(\R^*)^k$.
The $k$-cells of $(\R^*)^k$ are its connected components. Let $\mathcal I_k\subset C_0(T_k)$ be the augmentation
ideal. Let $q = -p$, so $0\leq q\leq k$. The vector space $C_0(T_k)$ has dimension $2^k$, and for each $q$ the
quotient $\mathcal I^q/\mathcal I^{q+1}$ has dimension $\binom{k}{q}$. A basis for $\mathcal I^q/\mathcal I^{q+1}$
can be defined as follows. Let $t_1,\dots,t_k$ be the standard generators of the multiplicative group $T_k$, 
$$
t_i = (t_{i1},\dots,t_{ik}),\  t_{ij} = \begin{cases} -1 \qquad  i = j  \\ +1 \qquad i \neq j \\ \end{cases}
$$

If $S\subset \{1,\dots,k\}$, let $T_S$ be the subgroup of $T_k$ generated by $\{t_i\ ;\ i\in S\}$, and define 
$[T_S]\in C_0(T_k)$ by 
$$
[T_S] = \sum_{t\in T_S}[t].
$$
 Then $\{[T_S]\ ; |S| = q\}$ is a basis for $\mathcal I^q/\mathcal I^{q+1}$ (see \cite{BFMH}).

To prove that $\alpha((\mathcal I_k)^q)\subset \mathcal N_{-q}C_k((\R^*)^k)$ we just need to show that if $|S| = q$ then
$\alpha([T_S])\in \mathcal N_{-q}C_k((\R^*)^k)$. Now the chain $\alpha([T_S])\in C_k((\R^*)^k)$ is represented by the
semialgebraic set $A_S\subset (\R^*)^k$, 
$$
A_S= \{(x_1,\dots,x_k)\ ;\ x_i > 0,\ i\notin S \},
$$
 and $ \varphi = 2^{k-q}\1_{A_S}$ is Nash constructible. To see this consider the compactification
  $(\mathbb P^1(\R))^k$ of $(\R^*)^k$.
 We have $\varphi = \tilde\varphi|(\R^*)^k$, where $\tilde\varphi = f_*\1_{(\mathbb P^1(\R))^k}$, with 
 $f:(\mathbb P^1(\R))^k \to (\mathbb P^1(\R))^k$ defined as follows. If $z = (u:v) \in \mathbb P^1(\R)$, let $f_1(z) = (u:v)$, and $f_2(z) = (u^2:v^2)$. Then
$$
f(z_1,\dots, z_k) = (w_1,\dots, w_k),\ w_i = \begin{cases} f_1(z_i) \qquad i\in S \\ f_2(z_i) \qquad i\notin S\\ \end{cases}
$$
This completes the proof.
\end{proof}

\begin{lem}\label{local}
Let $\sigma$ be a codimension $k$ cone of $\Delta$, and let
$$
C_i(\sigma) = \begin{cases} C_k(\sigma) \qquad i = k \\ 0 \qquad \qquad i\neq k \\ \end{cases}
$$
For all $p\leq 0$,
$$
\alpha_*: H_*(\mathcal T_pC_*(\sigma))\to H_*(\mathcal N_pC_*(\mathcal O_\sigma))
$$
is an isomorphism.
\end{lem}
\begin{proof}
Again we only need to consider the case $\mathcal O_\sigma = (\R^*)^k$, where
$\sigma$ is the trivial cone $0$ in $\R^n$. Now
$$
H_i(C_*(0)) = \begin{cases} C_k(0) \qquad i = k \\ 0 \qquad\qquad i\neq k\\ \end{cases}
$$
and
$$
H_i(C_*((\R^*)^k)) = \begin{cases} \Ker \partial_k \qquad i = k\\ 0 \qquad \qquad i\neq k\\ \end{cases}
$$
where $\partial_k: C_k((\R^*)^k)\to C_{k-1}((\R^*)^k)$. The vector space $\Ker \partial_k$ has
basis the cycles represented by the components of $(\R^*)^k$, and $\alpha:C_k(0)\to C_k((\R^*)^k)$
is a bijection from the cells of $C_k(0)$ to the components of $(\R^*)^k$. Thus $\alpha:C_k(0)\to
\Ker\partial_k$ is an isomorphism of vector spaces. Therefore $\alpha$ takes the basis
$\{A_S\ ;\ |S|=q \}_{q = 0,\dots,k}$ to a basis of $\Ker\partial_k$. The proof of Proposition \ref{filtrations}
shows that if $|S|\geq q$ then $A_S\in \mathcal N_{-q}C_k((\R^*)^k)$. We claim further that
if $|S|<q$ then $A_S\notin \mathcal N_{-q}C_k((\R^*)^k)$. It follows that $\{A_S\ ;\ |S| \geq q\}$
is a basis for $H_k(N_{-q}C_*((\R^*)^k)$,  and so 
$$
\alpha_*: H_*(\mathcal T_{-q}C_*(0))\to H_*(\mathcal N_{-q}C_*((\R^*)^k))
$$
is an isomorphism, as desired.

To prove the claim, it suffices to show that if $\bar A_S$ is the closure of $A_S$ in $\R^n$, then
$\bar A_S\notin \mathcal N_{-q}C_k((\R^*)^k)$. We show this by induction on $k$. The case $k=1$
is clear: If $\bar A = \{x\ ;\ x\geq 0\}$ then $\bar A\notin \mathcal N_{-1}C_1(\R)$ because
$\partial \bar A\neq 0$. In general $\bar A_S = \{(x_1,\dots,x_k)\ ;\ x_i \geq 0, i\notin S \}$.
Suppose $\bar A_S$ is $(-q)$-Nash constructible for some $q>|S|$. Then there exists
$\varphi:\R^k\to 2^{k-q}\Z$ generically Nash constructible in dimension $k$ such that
$$
\bar A_S = \{x\in \R^k\ ;\ \varphi(x)\notin 2^{k-q+1}\Z\},
$$
up to a set of dimension $< k$.
Let $j\notin S$, and let $W_j=\{(x_1,\dots,x_k)\ ;\ x_j=0\}\cong\R^{k-1}$. Then $\partial_{W_j}\varphi:W_j\to
2^{k-q-1}\Z$, and $\bar A_S\cap W_j = \{x\in W_j\ ;\ \partial_{W_j}\varphi(x)\notin 2^{k-q}\Z\}$,
up to a set of dimension $<k-1$.
Hence  $\bar A_S\cap W_j\in \mathcal N_{-q}C_{k-1}(W_j)$.
But $\bar A_S\cap W_j = \{(x,\dots, x_k)\ ;\ x_j = 0, x_i\geq 0, i\notin S \}$, and so
by inductive hypothesis $\bar A_S\cap W_j\notin \mathcal N_{-q}C_{k-1}(W_j)$, which is a contradiction.
\end{proof}
\begin{lem}\label{global}
For every toric variety $X_\Delta$ and every $p\leq 0$,
$$
\alpha_*:H_*(\mathcal T_pC_*(\Delta))\to H_*(\mathcal N_pC_*(X_\Delta))
$$
is an isomorphism.
\end{lem}
\begin{proof}
We show by induction on orbits that the lemma is true for every variety $Z$ that is a union
of orbits in the toric variety $X_\Delta$. Let $\Sigma$ be a subset of $\Delta$, and let $\Sigma'=\Sigma \setminus \{\sigma\}$,
where $\sigma\in \Sigma$ is a minimal cone, \emph{i.\ e.}\ there is no $\tau\in \Sigma$ with $\tau<\sigma$.
Let $Z$, resp.\ $Z'$, be the union of the orbits corresponding to cones in $\Sigma$, resp.\ $\Sigma'$.
Then $Z'$ is closed in $Z$, and $Z\setminus Z' = \mathcal O_\sigma$. We have a commutative diagram
with exact rows:
\begin{equation*}
\minCDarrowwidth 1pt
\begin{CD}
\cdots@>>>H_i(\mathcal T_pC_*(\Sigma')) @>>> H_i(\mathcal T_pC_*(\Sigma)) @>>>  H_i(\mathcal T_pC_*(\sigma))@>>> H_{i-1}(\mathcal T_pC_*(\Sigma'))@>>>\cdots \\
@. @VV\beta_i V
 @VV\gamma_i V @VV\alpha_iV @VV\beta_{i-1} V @. \\
\cdots@>>>H_i(\mathcal N_pC_*(\Sigma')) @>>> H_i(\mathcal N_pC_*(\Sigma)) @>>>  H_i(\mathcal N_pC_*(\sigma))@>>> H_{i-1}(\mathcal N_pC_*(\Sigma'))@>>>\cdots 
\end{CD}
\end{equation*}  
By  Lemma \ref{global} $\alpha_i$ is an isomorphism for all $i$. By inductive hypothesis $\beta_i$ is an isomorphism
for all $i$. Therefore $\gamma_i$ is an isomorphism for all $i$.
\end{proof}
\begin{thm}
For every toric variety $X_\Delta$ and every $p\leq 0$,
$$
\alpha_*:H_*\left(\frac{\mathcal T_pC_*(\Delta)}{\mathcal T_{p-1}C_*(\Delta)}\right)\to H_*\left(\frac{\mathcal N_pC_*(X_\Delta)}{\mathcal N_{p-1}C_*(X_\Delta)}\right)
$$
is an isomorphism.
\end{thm}
\begin{proof}
This follows from Lemma \ref{global} and the  long exact homology sequences of  the pairs $(\mathcal T_pC_*(\Delta),\mathcal T_{p-1}C_*(\Delta))$ 
and $(\mathcal N_pC_*(X_\Delta),\mathcal N_{p-1}C_*(X_\Delta))$.
\end{proof}
Thus for every toric variety $X_\Delta$ the toric filtered complex $\mathcal TC_*(\Delta)$ is quasi-isomorphic to the Nash constructible filtered complex $\mathcal NC_*(X_\Delta)$, and so the toric spectral sequence \cite{BFMH} is isomorphic to the weight spectral sequence.

\example For toric varieties of dimension at most 4, the toric spectral sequence collapses (\cite{BFMH}, \cite{sine}).
V. Hower \cite{hower} discovered that the spectral sequence does not collapse for the 6-dimensional 
projective toric variety associated to the matroid of the Fano plane.


\bigskip
\section{Appendix: Semialgebraic chains}
\label{chains}

In this appendix we denote by $X$ a locally compact semialgebraic set (\emph{i.e.} a semialgebraic 
subset of the set of real points of a real algebraic variety) and by  $C_*(X)$ the complex of semialgebraic
chains of $X$ with closed supports and coefficients in $\Z_2$.  The complex $C_*(X)$  has the following
geometric description, which is equivalent to the usual definition using a
semialgebraic triangulation (\cite{BCR} 11.7).

A \emph{semialgebraic chain} $c$ of $X$ is an equivalence class of closed
semialgebraic subsets of $X$. For $k\geq 0$, let $S_k(X)$ be the $\Z_2$
vector space generated by the closed semialgebraic subsets of $X$ of
dimension $\leq k$. Then $C_k(X)$ is the $\Z_2$ vector space obtained as
the quotient of $S_k(X)$ by the following relations:
\begin{enumerate}
\item If $A$ and $B$ are
closed semialgebraic subsets of $X$ of dimension at most $k$, then
$$
A+B\sim \text{cl}(A\div B),
$$
where $A\div B = (A\cup B)\setminus (A\cap B)$ is the symmetric 
difference of $A$ and $B$, and cl denotes closure.
\item
If $A$ is a closed semialgebraic subset of $X$ and $\dim A< k$, then $A\sim 0$.
\end{enumerate}
If the chain $c$ is represented by the semialgebraic set $A$, we write
$c = [A]$.
If $c\in C_k(X)$, the \emph{support} of $c$, denoted $\supp c$,
is the smallest closed semialgebraic set representing $c$. If $c = [A]$
then $\supp c = \{x\in A\ ;\ \dim_xA = k\}$.

The \emph{boundary} operator $\partial_k:C_k(X)\to C_{k-1}(X)$ can be defined 
using the link operator $\Lambda$ on constructible functions \cite{mccroryparusinski}.
If $c\in C_k(X)$ with $c=[A]$, then $\partial_k c=[\partial A]$, where
$\partial A= \{x\in A\ ;\ \Lambda\1_A(x) \equiv 1 \pmod 2\}$. The operator
$\partial_k$ is well-defined, and $\partial_{k-1}\partial_k = 0$, since
$\Lambda\circ\Lambda=2\Lambda$.

If $f:X\to Y$ is a proper continuous semialgebraic map,  the \emph{pushforward}
homomorphism $f_*:C_k(X)\to C_k(Y)$ is defined as follows. Let $A$ be
a representative of $c$. Then $f(A) \sim B_1+\cdots+B_l$, where each closed
semialgebraic set $B_i$ has the property that $\#(A\cap f^{-1}(y))$
is constant mod 2 on $B_i\setminus B_i'$ for some closed semialgebraic
set $B_i\subset B_i$ with $\dim B_i'<k$. For each $i$ let $n_i\in \Z_2$ be this constant
value. Then
$f_*(c) = n_1[B_1]+\cdots+n_l[B_l]$.

Alternately, $f_*(c) = [B]$, where $B=\text{cl}\{y\in Y \ ;\ f_*\1_A(y) \equiv 1\pmod 2\}$,
and $f_*$ is pushforward for constructible functions \cite{mccroryparusinski}. From
this definition it is easy to prove the standard properties $g_*f_* = (gf)_*$
and $\partial_k f_* = f_*\partial_k$.

We use two basic operations on semialgebraic chains: restriction and closure. These
operations do not commute with the boundary operator in general.

Let $c\in C_k(X)$ and let $Z\subset X$ be a locally closed semialgebraic subset. If $c=[A]$, we define the
\emph{restriction}  by $c|_Z = [A\cap Z]\in C_k(Z)$. This operation is well-defined. If $U$ is an open
semialgebraic subset of $X$, then $\partial_k(c|_U)=(\partial_k c)|_U$. 

Now let $c\in C_k(Z)$ with $Z\subset X$ locally closed semialgebraic. If $c=[A]$ we define the \emph{closure}
 by $\bar c = [\text{cl}(A)]\in C_k(X)$, where $\text{cl}(A)$ is the closure of $A$ in $X$. Closure
is a well-defined operation on semialgebraic chains.

By means of the restriction and closure operations, we define the pullback of a chain in the following
situation, which can be applied to an acyclic square (\ref{acyclic}) of real algebraic varieties.  
Consider a square of 
locally closed semialgebraic sets, 
\begin{equation*}\minCDarrowwidth 1pt\begin{CD}
\tilde Y @> >> \tilde X \\
 @V VV @VV\pi V \\
Y @>i >>X
\end{CD}\end{equation*}such that  $\pi:\tilde X\to X$ is a proper continuous semialgebraic map, $i$ is the inclusion of a closed semialgebraic subset, $\tilde Y = \pi^{-1}(Y)$, and  the restriction
of $\pi$ is a homeomorphism $\pi': \tilde X \setminus \tilde Y \to X \setminus Y$.  
Let $c\in C_k(X)$. We define the \emph{pullback}
$\pi^{-1}c\in C_k(\tilde X)$ by the formula
$$
\pi^{-1}c = \overline{((\pi')^{-1})_*(c|_{X\setminus Y})}.
$$
Pullback
does not commute with the boundary operator in general.



\begin{thebibliography}{99}

\bibitem{AKMW} D. Abramovich, K. Karu, K. Matsuki, J. W\l odarczyk,
\emph{Torification and factorization of birational maps}, J. Amer.
Math. Soc. \textbf{29} (2002), 531--572.

\bibitem{akbking} S. Akbulut, H. King, \emph{The topology of real algebraic sets}, 
Enseign. Math. \textbf{29} (1983),
221--261.

\bibitem{akbulutking} S. Akbulut, H. King, \emph{Topology of
Real Algebraic Sets}, MSRI Publ. \textbf{25}, Springer Verlag,
New York, 1992.

\bibitem {BFMH} F. Bihan, M. Franz, C. McCrory, J. van Hamel,
\emph{Is every toric variety an M-variety?}, Manuscripta Math. 
\textbf{120} (2006), 217--232.

\bibitem {BCR}  J. Bochnak, M. Coste, M.-F. Roy,
\emph{Real Algebraic Geometry}, Springer Verlag, New York,
1992.


 \bibitem{bonnard2}  I. Bonnard,  \emph{Un crit\`ere pour reconaitre les fonctions alg\'ebriquement constructibles},  J. Reine Angew. Math. \textbf{526} (2000), 61--88.

\bibitem{bonnard}  I. Bonnard,  \emph{Nash constructible functions},
Manuscripta Math. \textbf{112} (2003), 55--75.






\bibitem{deligne2} P. Deligne, \emph{Th\'eorie de Hodge II}, IHES
Publ. Math. \textbf{40} (1971), 5--58.

\bibitem{deligne} P. Deligne, \emph{Poids dans la cohomologie des
vari\'et\'es alg\'ebriques}, Proc. Int. Cong. Math. Vancouver (1974),
79--85.





\bibitem{fichou}  G. Fichou,
\emph{Motivic invariants of arc-symmetric sets and blow-Nash equivalence},
  Compositio Math. \textbf{141} (2005) 655--688.

\bibitem{fulton} W. Fulton, \emph{Introduction to Toric Varieties},
Annals of Math. Studies \textbf{131}, Princeton, 1993.

\bibitem{gilletsoule} H. Gillet, C. Soul\'e, \emph{Descent,
motives, and K-theory}, J. Reine Angew. Math. \textbf{478} (1996),
127--176.

\bibitem{navarro} F. Guill\'en, V. Navarro Aznar, \emph{Un crit\`ere
    d'extension des foncteurs d\'efinis sur les sch\'emas lisses}, 
 IHES Publ. Math. \textbf{95} (2002), 1--83.
 
 \bibitem{navarro2} F. Guill\'en, V. Navarro Aznar, \emph{Cohomological
 descent and weight filtration} (2003). (Abstract: 
 http://congreso.us.es/rsme-ams/sesionpdf/sesion13.pdf.)



 \bibitem{hower} V. Hower, \emph{A counterexample to the maximality of toric varieties}, Proc. Amer. 
 Math. Soc., \textbf{136} (2008), 4139--4142.

\bibitem{kucharz}  W. Kucharz,
\emph{Homology classes represented by semialgebraic arc-symmetric
   sets}, Bull. London Math. Soc. \textbf{37} (2005), 514--524.


\bibitem{kurdyka1} K. Kurdyka,
\emph{Ensembles semi-alg\'ebriques
sym\'etriques par arcs}, Math. Ann. \textbf{ 281} (1988),  445--462.

\bibitem{kurdyka2}  K. Kurdyka, \emph{Injective endomorphisms of
real algebraic sets are surjective}, Math. Ann. \textbf{313} no.1 (1999),
69--83



\bibitem{aussois} K. Kurdyka, A. Parusi\'nski,
\emph{Arc-symmetric sets and arc-analytic mappings}, Panoramas \&
Syntheses \textbf{24}, Soc. Math. France (2007), 33--67.


\bibitem{maclane} S. MacLane, \emph{Homology}, Springer-Verlag, Berlin 1963.


\bibitem{marshall} M. A. Marshall,  \emph{Open questions in the theory of spaces of orderings},
J. Symbolic Logic \textbf{67} (2002), no. 1, 341--352. 


\bibitem{mccroryparusinski}  C. McCrory, A. Parusi\'nski,
\emph{Algebraically constructible functions},
Ann.  Sci. \'Ec. Norm. Sup.  \textbf{30} (1997), 527--552.

\bibitem{virtual}  C. McCrory, A. Parusi\'nski,
\emph{Virtual Betti numbers of real algebraic varieties},
Comptes Rendus Acad. Sci. Paris, Ser. I, \textbf{336} (2003),
763--768. (See also http://arxiv.org/pdf/math.AG/0210374.)



\bibitem{mikhalkin1}   G. Mikhalkin, \emph{Blowup equivalence of smooth 
closed manifolds},
Topology \textbf{36} (1997), 287--299.

\bibitem{mikhalkin2}   G. Mikhalkin, \emph{Birational equivalence for smooth 
manifolds with boundary},  (Russian) Algebra i Analiz 11 (1999), no. 5, 152--165; translation 
in St. Petersburg Math. J.  11 (2000), no. 5, 827--836

\bibitem{nagata} M. Nagata, \emph{Imbedding of an abstract variety in a complete variety},
J. Math. Kyoto U. \textbf{2} (1962), 1--10.





\bibitem{parusinski}  A. Parusi\'nski,
\emph{Topology of injective endomorphisms of real algebraic sets},
Math. Ann. \textbf{328} (2004), 353--372.

\bibitem{parusinskiszafraniec} A. Parusi\'nski, Z. Szafraniec, \emph{
Algebraically constructible functions and signs of polynomials}, Manuscripta Math. 
\textbf{93} (1997), no. 4, 443--456.

\bibitem{penn1}  H. Pennaneac'h,
\emph{Algebraically constructible chains}, Ann. Inst. Fourier (Grenoble) 
\textbf{51} (2001), no. 4, 939--994,

\bibitem{penn2}  H. Pennaneac'h,
\emph{Nash constructible chains}, preprint Universit\`a di Pisa, 
 (2003). 


\bibitem{penn3}  H. Pennaneac'h,
\emph{Virtual and non-virtual algebraic Betti numbers},
Adv. Geom. \textbf{5} (2005), no. 2, 187--193. 

\bibitem{peterssteenbrink} C. Peters, J. Steenbrink,
\emph{Mixed Hodge Structures}, Springer Verlag, Berlin-Heidelberg, 2008.

 \bibitem{sine} A. Sine, \emph{ Probl\`eme de maximalit\'e pour les vari\'et\'es toriques}, Th\`ese Doctorale, Universit\'e d'Angers 2007.

\bibitem{thom}  R. Thom, \emph{Quelques propri\'et\'es
globales des vari\'et\'es diff\'erentiables},
Comm. Math. Helv. \textbf{28}
(1954), 17--86.





\bibitem{totaro} B. Totaro, \emph{Topology of singular algebraic
varieties}, Proc. Int. Cong. Math. Beijing (2002), 533-541.



\end{thebibliography}
\end{document}